\documentclass[10pt,reqno]{amsart}

\usepackage{amsmath}
\usepackage{amsthm}
\usepackage{amssymb}
\usepackage{amsfonts}
\usepackage{amsxtra}
\usepackage[all]{xypic}
\usepackage{fullpage}
\usepackage{amscd}
\usepackage[noautoscale,enableskew]{youngtab}

\newcommand{\ff}{\footnote}

\let\nc\newcommand

\theoremstyle{plain}
\newtheorem*{thm}{Theorem}
\newtheorem*{prop}{Proposition}
\newtheorem*{cor}{Corollary}
\newtheorem*{lem}{Lemma}
\newtheorem*{conjecture}{Conjecture}
\theoremstyle{definition}
\newtheorem*{defn}{Definition}

\newtheorem*{remark}{Remark}

\nc{\bdm}{\begin{displaymath}}
\nc{\edm}{\end{displaymath}}
\nc{\bthm}{\begin{thm}}
\nc{\ethm}{\end{thm}}
\nc{\blem}{\begin{lem}}
\nc{\elem}{\end{lem}}
\nc{\bcor}{\begin{cor}}
\nc{\ecor}{\end{cor}}
\nc{\beq}{\begin{equation}}
\nc{\eeq}{\end{equation}}
\nc{\bprop}{\begin{prop}}
\nc{\eprop}{\end{prop}}
\nc{\bdefn}{\begin{defn}}
\nc{\edefn}{\end{defn}}

\nc{\Z}{\mathbb{Z}}
\newcommand{\N}{\mathbb{N}}
\newcommand{\Q}{\mathbb{Q}}

\newcommand{\C}{\mathbb{C}}
\newcommand{\h}{\mathfrak{h}}

\renewcommand{\Im}{\mbox{\textrm{Im}}\,}

\newcommand{\rad}{\mbox{\textrm{rad}}\,}

\nc{\Hom}{\textrm{Hom}}
\nc{\rank}{\textrm{rank} \,}
\nc{\ds}{\dots}

\let\mc\mathcal
\let\mf\mathfrak

\nc{\HW}{\bar{H}_{\mathbf{c}}(W)}
\nc{\HK}{\bar{H}_{\mathbf{c}}(K)}
\nc{\HtK}{\widetilde{H}_{\mathbf{c}}(K)}
\nc{\CMW}{\textsf{CM}_{\mbf{c}}(W)}
\nc{\CMK}{\textsf{CM}_{\mbf{c}}(K)}
\nc{\mbf}{\mathbf}
\nc{\LK}{\textsf{Irr}(K)}
\nc{\LW}{\textsf{Irr}(W)}
\nc{\Res}{\mathsf{Res} \, }
\nc{\Ind}{\mathsf{Ind} \, }

\nc{\cont}{\textrm{cont}}
\renewcommand{\mod}{\textrm{mod}}
\nc{\eWb}{\mathbf{e}_{W_b}}
\nc{\eW}{\mathbf{e}_{W}}
\nc{\msf}{\mathsf}
\nc{\Ui}{\mc{U}_{i,+}}
\nc{\Uone}{\mc{U}_{1,+}}
\nc{\Utwo}{\mc{U}_{2,+}}
\newcommand{\mmod}{\text{-}\mathrm{mod}}
\nc{\minusone}{-1}
\nc{\minustwo}{-2}
\nc{\p}{\partial}

\begin{document}

\title{The Calogero-Moser partition for $G(m,d,n)$}

\author{Gwyn Bellamy}

\address{School of Mathematics and Maxwell Institute for Mathematical Sciences, University of Edinburgh, James Clerk Maxwell Building, Kings Buildings, Mayfield Road, Edinburgh EH9 3JZ, Scotland}
\email{G.E.Bellamy@sms.ed.ac.uk (current address Gwyn.Bellamy@Manchester.ac.uk)}

\begin{abstract}
\noindent We show that it is possible to deduce the ``Calogero-Moser partition'' of the irreducible representations of the complex reflection groups $G(m,d,n)$ from the corresponding partition for $G(m,1,n)$. This confirms, in the case $W = G(m,d,n)$, a conjecture of Gordon and Martino relating the Calogero-Moser partition to Rouquier families for the corresponding cyclotomic Hecke algebra.   
\end{abstract}

\maketitle

\section{Introduction}

\subsection{} Let $W$ be a finite complex reflection group. Associated to $W$ is a family of noncommutative algebras, the rational Cherednik algebras. These algebras depend on a pair of parameters, $t$ and $\mbf{c}$ (precise definitions are given in section \ref{subsection:defns}). At $t = 0$ the algebras are finite modules over their centres. The aim of this paper is to continue the study of a certain finite dimensional quotient of the rational Cherednik algebra at $t = 0$, the restricted rational Cherednik algebra. The blocks of the restricted rational Cherednik algebra induce a partitioning of the set $\LW$ of irreducible $W$-modules, called the Calogero-Moser partition. Using the geometry of certain quiver varieties, Gordon and Martino \cite{12} have given an explicit combinatorial description of the Calogero-Moser partition when $W = C_m \wr S_n$. We show that Clifford theoretic arguments can be use to extend this result to the normal subgroups $G(m,d,n)$ of $C_m \wr S_n$. In their paper \cite{12}, Gordon and Martino conjecture that the Calogero-Moser partition should be related, in some precise way, to the Rouquier blocks of a particular Hecke algebra associated to the same complex reflection group $W$. This conjecture is refined in \cite{Mo} and, by comparing the combinatorial description of these partitions, is shown to be true when $W = C_m \wr S_n$. A consequence of the main result of this paper is that the conjecture as stated in \cite[Conjecture 2.7 (i)]{Mo}, is true for all $G(m,d,n)$. However it is important to note here that, when $n = 2$ and $d$ is even, there are certain ``unequal parameter'' cases where our methods fail (see (\ref{sec:unequal}) for details). In these cases it is not known what the Calogero-Moser partition is.\ff{Mathematics Subject Classification (2010) 16G99,05E10.}    

\section{The rational Cherednik algebra at $t = 0$}

\subsection{Definitions and notation}\label{subsection:defns}

Let $W$ be a complex reflection group, $\mathfrak{h}$ its reflection representation over $\C$ with rank $\mathfrak{h} = n$, and $\mathcal{S}(W)$ the set of all complex reflections in $W$. Let $( \cdot, \cdot ) : \mathfrak{h} \times \mathfrak{h}^* \rightarrow \C$ be the natural pairing defined by $(y,x) = x(y)$. For $s \in \mathcal{S}(W)$, fix $\alpha_s \in \mathfrak{h}^*$ to be a basis of the one dimensional space $\Im (s - 1)|_{\mathfrak{h}^*}$ and $\alpha_s^{\vee} \in \mathfrak{h}$ a basis of the one dimensional space $\Im (s - 1)|_{\mathfrak{h}}$, normalised so that $\alpha_s(\alpha_s^\vee) = 2$. Choose $\mathbf{c} : \mathcal{S}(W) \rightarrow \C$ to be a $W$-equivariant function and $t$ a complex number. The \textit{rational Cherednik algebra}, $H_{t,\mathbf{c}}(W)$, as introduced by Etingof and Ginzburg \cite[page 250]{1}, is the quotient of the skew group algebra of the tensor algebra, $T(\mf{h} \oplus \mf{h}^*) \rtimes W$, by the ideal generated by the relations

\begin{equation}\label{eq:rel}
[x_1,x_2] = 0, \qquad [y_1,y_2] = 0, \qquad [x_1,y_1] = t (y_1,x_1) - \sum_{s \in \mathcal{S}} \mathbf{c}(s) (y_1,\alpha_s)(\alpha_s^\vee,x_1) s, 
\end{equation}
\noindent for all $x_1,x_2 \in \mathfrak{h}^* \textrm{ and } y_1,y_2 \in \mathfrak{h}$.\\

\noindent For any $\nu \in \C \backslash \{ 0 \}$, the algebras $H_{\nu t,\nu \mathbf{c}}(W)$ and $H_{t,\mathbf{c}}(W)$ are isomorphic. In this article we will only consider the case $t = 0$, therefore we are free to rescale $\mbf{c}$ by $\nu$ whenever this is convenient.\\

\noindent A fundamental result for rational Cherednik algebras, proved by Etingof and Ginzburg \cite[Theorem 1.3]{1}, is that the PBW property holds for all $t, \mbf{c}$. That is, there is a vector space isomorphism 
\beq\label{eq:PBW}
H_{t, \mbf{c}}(W) \stackrel{\sim}{\rightarrow} \C [\h] \otimes \C W \otimes \C [\h^*].
\eeq

\subsection{The restricted rational Cherednik algebra}\label{sec:restricteddefinition} Let us now concentrate on the case $t = 0$, and we omit $t$ from the notation. In this case the algebra $H_{\mbf{c}}(W)$ is a finite module over its centre $Z_{\mbf{c}}(W)$. By \cite[Proposition 4.5]{1}, we have an inclusion $A = \C[\mf{h}]^W\otimes\C[\mf{h}^*]^W \subset Z_{\mathbf{c}}$. This allows us to define the \textit{restricted rational Cherednik algebra} $\HW$ as

\begin{displaymath}
\HW = \frac{H_{\mathbf{c}}(W)}{A_+H_{\mathbf{c}}(W)},
\end{displaymath}

\noindent where $A_+$ denotes the ideal in $A$ of elements with zero constant term. The PBW property (\ref{eq:PBW}) implies that $\HW \cong \C [\h]^{co W} \otimes \C W \otimes \C [\h^*]^{coW}$ as vector spaces, where $\C [\h]^{co W} := \C[ \h] / \langle \C[\h]^W_+ \rangle$ is the coinvariant ring of $\h$ with respect to $W$. In particular, $\dim \HW = |W|^3$. The inclusion $\C[\mathfrak{h}]^W \otimes \C[\mathfrak{h}^*]^W \hookrightarrow Z_{\mathbf{c}}(W)$ defines a surjective, finite morphism $\Upsilon \,  : \,\textrm{Spec}\, (Z_{\mathbf{c}}(W)) \twoheadrightarrow \h^* / W \times \h/W$.

\subsection{The Calogero-Moser partition}\label{sec:partitions}
Fix a complete set of non-isomorphic simple $W$-modules and denote it by $\LW$. Following \cite{12} we define the \textit{Calogero-Moser partition} of $\textrm{Irr} \, \HW$ to be the set of equivalence classes of $\textrm{Irr} \,  \HW$  under the equivalence relation $L \sim M$ if and only if $L$ and $M$ belong to the same block of $\HW$. The set of equivalence classes will be denoted $\CMW$. It has been shown, \cite[Proposition 4.3]{Baby}, that $\textrm{Irr} \, \HW$ can be naturally identified with $\LW$. Thus the Calogero-Moser partition $\CMW$ will be thought of a partition of $\textrm{Irr}(W)$ throughout this article. Given $\lambda, \mu \in \LW$ we say that $\lambda, \mu $ belong to the same partition of $\CMW$ if they are in the same equivalence class. 

\section{Blocks of normal subgroups}

\subsection{} Throughout this section we fix an irreducible complex reflection group $W$ with reflection representation $\h$. Moreover we assume that there exists a normal subgroup $K \triangleleft W$ such that $K$ acts, via inclusion in $W$, on $\h$ as a complex reflection group (though $\h$ need not be irreducible as a $K$-module) and that $W/K \cong C_d$, the cyclic group of order $d$. Since $K$ is normal in $W$, the group $W$ acts on $\mc{S}(K)$ by conjugation. Let us fix a $W$-equivariant function $\mbf{c} \,  : \, \mc{S}(K) \rightarrow \C$. We extend this to a $W$-equivariant function $\mbf{c} \, : \, \mc{S}(W) \rightarrow \C$ by setting $\mbf{c}(s) = 0$ for $s \in  \mc{S}(W) \backslash  \mc{S}(K)$. Note that the partition of $ \mc{S}(K)$ into $K$-orbits can be finer than the corresponding partition into $W$-orbits. Thus a $K$-equivariant function on $\mc{S}(K)$ is not always $W$-equivariant. However, as will be shown below, this problem does not occur in the cases we consider. For our choice of parameter $\mbf{c}$, the defining relations (\ref{eq:rel}) show that the natural map $T (\h \oplus \h^*) \rtimes K \rightarrow H_{t,\mbf{c}}(W)$ descends to an algebra morphism $H_{t,\mbf{c}}(K) \rightarrow H_{t,\mbf{c}}(W)$. The PBW property (\ref{eq:PBW}) shows that this map is injective.

\begin{prop}
For $\mbf{c}$ as defined above, the algebra $H_{t,\mbf{c}}(K)$ is a subalgebra of $H_{t,\mbf{c}}(W)$. 
\end{prop}

\subsection{} As explained in the introduction, the goal of this article is to relate the Calogero-Moser partition of $K$ to the Calogero-Moser partition of $W$. However the algebra $\HK$ is not a subalgebra of $\HW$. To overcome this we study an intermediate algebra, $\HtK$, which is defined to be the image of $H_{\mbf{c}}(K)$ in $\HW$. Thus we are in the following setup:
\begin{displaymath}
\vcenter{
\xymatrix{
H_{0,\mbf{c}}(K) \ar@{->>}[d] \ar@{->}[r] & H_{0,\mbf{c}}(W) \ar@{->>}[d] \\
\HtK \ar@{->>}[d] \ar@{->}[r] & \HW \\
\HK & }
}
\end{displaymath}
where the horizontal arrows are inclusions. To be precise, $\HtK := H_{0,\mbf{c}}(K) / A_+ \cdot H_{0,\mbf{c}}(K)$, where $A = \C[\h]^W \otimes \C[\h^*]^W$ and $A_+$ the ideal of polynomials with constant term zero. The PBW property (\ref{eq:PBW}) implies that $\HtK \cong \C[ \h]^{co W} \otimes \C K \otimes \C [\h^*]^{co W}$ and hence has dimension $|K| \cdot |W|^2$. The idea is to relate the block partition of $\HtK$ to $\CMW$ via the formalism of twisted symmetric algebras. The Proposition below shows that this allows us to deduce information about the partition $\CMK$.

\subsection{}\label{thm:lifting} As noted in (\ref{sec:partitions}), the set $\{ L(\lambda) \, | \, \lambda \in \LK \}$ is a complete set of non-isomorphic simple modules for $\HK$. There is a natural surjective map $\HtK \twoheadrightarrow \HK$ and the kernel of this map is generated by certain central nilpotent elements of $\HtK$. Therefore the kernel is contained in the radical of $\HtK$. This implies that $\{ L(\lambda) \, | \, \lambda \in \LK \}$ is also a complete set of non-isomorphic simple modules for $\HtK$ and the block partition of $\HtK$ corresponds to a partition of the set $\LK$. In particular, the space $L(\lambda)$ is both a simple $\HK$ and $\HtK$-module. However when we wish to consider $L(\lambda)$ as a $\HtK$-module we will denote it by $\tilde{L}(\lambda)$.

\begin{prop}
The Calogero-Moser partition $\CMK$ of $\LK$ and the block partition of $\HtK$ on $\LK$ are equal because the blocks of $\HtK$ are the preimages of the blocks of $\HK$ under the natural map $\HtK \twoheadrightarrow \HK$.
\end{prop}

\begin{proof}
Let us again denote by $A$ the algebra $\C[\h]^W \otimes \C[\h^*]^W$ and define $B = \C[\h]^K \otimes \C[\h^*]^K$. Then we have inclusions $A \subset B \subset Z(H_\mbf{c}(K)) \subset H_\mbf{c}(K)$. The Proposition will follow from an application of a result of B. M\"uller; the version which we use here is stated in \cite[Proposition 2.7]{Ramifications}. Recall that $A_+$ is the maximal ideal of elements with constant term zero in $A$. Let $B_+$ be the maximal ideal of elements with constant term zero in $B$. Fix $Z := Z(H_\mbf{c}(K))$ and $H := H_\mbf{c}(K)$. M\"uller's Theorem says that the primitive central idempotents of $H / A_{+} \cdot H$ are the images of the primitive idempotents of $Z / A_{+} \cdot Z$, and similarly the primitive central idempotents of $H / B_{+} \cdot H$ are the images of the primitive idempotents of $Z / B_{+} \cdot Z$. However $A_+ \cdot Z \subset B_+ \cdot Z$ and $B_+ \cdot Z / A_+ \cdot Z$ is a nilpotent ideal in $Z / A_+ \cdot Z$; therefore the primitive idempotents of $Z / B_+ \cdot Z$ are the images of the primitive idempotents of $Z  / A_+ \cdot Z$. This implies that the primitive central idempotents of $H / B_+ \cdot H$ are the images of the primitive central idempotents of $H  / A_+ \cdot H$. This is equivalent to the statement of the Proposition.
\end{proof}

\subsection{}\label{lem:algpoly} The following lemma will be required later. 

\begin{lem}
We can choose a set $\{ f_1, \ds, f_n \}$ of homogeneous, algebraically independent generators of $\C[\mathfrak{h}]^K$ and positive integers $a_1, \ds, a_n$ such that $\{ f_1^{a_1}, \ds, f_n^{a_n} \}$ is a set of homogeneous, algebraically independent generators of $\C[\mathfrak{h}]^W$ and $a_1 \cdots a_n = d$.
\end{lem}

\begin{proof}
The ring $\C[\h]^K$ is $\N$-graded with $(\C[\h]^K)_0 = \C$. Therefore $\mf{m} := \C[\h]^K_+$, the ideal of polynomials with zero constant terms, is the unique maximal, graded ideal of $\C[\h]^K$. The group $W$ acts on $\mf{m}$ and hence also on $\mf{m}^2$. Let $U$ be a homogeneous, $W$-stable complement to $\mf{m}^2$ in $\mf{m}$. By \cite[Lemme 2.1]{BessisBonnafeRouquier}, $U$ generates $\C[\h]^K$ and so $\C[\h]^K = \C[U^*]$. The action of $W$ on $U^*$ factors through $C_d$. Since $\C[U^*]^{C_d} = \C[\mathfrak{h}]^W$ is a polynomial ring, the Chevalley-Shephard-Todd Theorem, \cite[Theorem 1.2]{CohenReflections}, says that $C_d$ acts on $U^*$ as a complex reflection group. Therefore we can decompose $U$ into a direct sum of one-dimensional, homogeneous $C_d$-modules, $U = \oplus_{i = 1}^n \C \cdot f_i$, and $C_d = C_{a_1} \times \cdots \times C_{a_n}$ such that the action of $C_d$ on $\C \cdot f_i$ factors through $C_{a_i}$ (with $C_{a_i}$ acting faithfully on $\C \cdot f_i$). Then $\C[\h]^W = \C[f_1^{a_1}, \ds, f_n^{a_n}]$ and the fact that $\C[\h]^W$ is a polynomial ring in $n$ variables means that the polynomials $f_1^{a_1}, \ds, f_n^{a_n}$ are algebraically independent. 
\end{proof} 

\begin{remark}  For $W = G(m,1,n)$ and $K = G(m,d,n)$ (as defined in Section \ref{sec:example}) we can make an explicit choice of invariant polynomials as described in Lemma \ref{lem:algpoly}. Let $e_i(x_1, \dots, x_n)$ denote the $i^{th}$ elementary symmetric polynomial in $x_1, \dots , x_n$. By \cite[page 387]{CohenReflections}, the following are a choice of algebraically independent, homogeneous generators for $\C [\mathfrak{h}]^W$:
\begin{displaymath}
e_i(x_1^m, \dots , x_n^m), \quad 1 \le i < n \qquad \textrm{and} \qquad (x_1 \dots x_n)^{mn}.
\end{displaymath}
In Lemma \ref{lem:algpoly} we take $f_n$ to be $(x_1 \dots x_n)^{\frac{nm}{d}}$ and $f_i = e_i(x_1^m, \dots , x_n^m)$ for $1 \le i < n$ so that $a_i = 1$ for $1 \le i < n$ and $a_n = d$.  
\end{remark}

\section{Automorphisms of rational Cherednik algebras}

\subsection{} The group $W$ is a finite subgroup of $GL(\h)$. Let us choose an element $\sigma \in N_{GL(\h)}(W) \subset GL(\h)$. Then $\sigma$ is an automorphism of $W$ and we can regard it as an algebra automorphism of $\C W$ by making $\sigma$ act trivially on $\C$. Moreover $\sigma$ acts naturally on $\h^*$ as $(\sigma \cdot x)(y) = x(\sigma^{-1} \cdot y)$ for $x \in \h^*$ and $y \in \h$. Therefore $\sigma$ also acts on $\C[\h^*]$ and $\C[\h]$. Let us explicitly write $\mc{S}(W) = \{ C_1, \ds , C_k \}$ for the set of conjugacy classes of reflections in $W$. Then $\sigma$ permutes the $C_i$'s and regarding $\sigma$ as an element of the symmetric group $S_k$ we write $\sigma \cdot C_i = C_{\sigma(i)}$. It can be checked from the defining relations (\ref{eq:rel}) that the maps
\bdm
x \mapsto \sigma(x), \qquad y \mapsto \sigma(y), \qquad w \mapsto \sigma(w), \qquad x \in \h^*, y \in \h, w \in W
\edm
define an algebra isomorphism
\bdm
\sigma \, : \, H_{t, \mbf{c}}(W) \stackrel{\sim}{\longrightarrow} H_{t,\sigma(\mbf{c})}(W),
\edm 
where $\sigma(\mbf{c}) = \sigma(c_1 , \ds , c_k) = (c_{\sigma^{-1}(1)}, \ds , c_{\sigma^{-1}(k)})$. Since $\sigma$ normalizes $W$, there is a well defined action of $\sigma$ on $\C[\h]^W \otimes \C[\h^*]^W$. Hence $\sigma$ descends to an isomorphism $\sigma \, : \, \HW \stackrel{\sim}{\rightarrow} \bar{H}_{\sigma(\mbf{c})}(W)$.

\subsection{} Now let us consider $K$. By definition $W \subset N_{GL(\h)}(K)$, therefore elements of $W$ act as isomorphisms between the various rational Cherednik algebras associated to $K$. Moreover, if we once again make the assumption that the parameter $\mbf{c}$ is $W$-equivariant then the elements of $W$ actually define automorphisms of $H_{t,\mbf{c}}(K)$. These induce automorphisms of $\HK$ and $\HtK$. Let $M$ be a module for one of the three algebras $\C K,\HK$ or $\HtK$. Then ${}^\sigma M$ is also a module for that algebra, where $M = {}^\sigma M$ as vector spaces and if $a$ is an element of the algebra and $m \in M$, then $a \cdot_{\sigma} m = \sigma^{-1}(a) \cdot m$. The following lemma is standard.

\begin{lem}
Let $\lambda$ be a $K$-module and $\sigma \in W$. Then ${}^\sigma L(\lambda) \cong L({}^\sigma \lambda)$ and ${}^\sigma \tilde{L}(\lambda) \cong \tilde{L}({}^\sigma \lambda)$.
\end{lem}

\subsection{Clifford theory}\label{sec:Clifford} We now define an action of $C_d$ on $\HtK$. For $\eta \in C_d$, choose a lift $\sigma$ of $\eta$ in $W$ and let $\lambda \in \LK$. Define 
\bdm
\eta \cdot \lambda = {}^\sigma \lambda, \quad \eta \cdot \tilde{L}(\lambda) = {}^\sigma \tilde{L}(\lambda).
\edm
Note that the action of $C_d$ is only well-defined up to isomorphism, therefore $C_d$ can be considered as acting on the isomorphism classes of the objects in $\HtK \mmod$. Given $\mu \in \LK$, the stabilizer subgroup of $C_d$ with respect to $\mu$ will be denoted $C_\mu$. Let $C_d^\vee = \Hom_{\textrm{gp}} (C_d, \C^*)$ be the group of characters of $C_d$. There is an action of $C_d^\vee$ on the isomorphism classes of the objects in $\HW \mmod$. First let us define an action of $C_d^\vee$ on $\LW$: $\delta \cdot \lambda = \lambda \otimes  \delta$, for $\delta \in C^\vee_d$ and $\lambda \in \LW$. The stabilizer subgroup of $C_d^\vee$ with respect to $\lambda$ will be denoted $C^\vee_\lambda$. We choose coset representatives $w_1, \ds , w_d$ of $C_d$ in $W$, then Lemma \ref{lem:algpoly} implies that $\HW = \bigoplus_{i} \HtK w_i$. Given a $\HW$-module $M$ we define $\delta \cdot M = M \otimes \delta$ with action
\bdm
h w_i \cdot( m \otimes \delta) = \delta(K w_i) (h w_i \cdot m) \otimes \delta.
\edm
This action does not depend on the choice of coset representatives and one can define $\delta$ as a functor on $\HW \mmod$, though we will not require this level of generality. 

\subsection{}\label{prop:Clifford} Let $\Res^W_K$ and $\Ind_K^W$ be the induction and restriction functors $\C K \mmod \leftrightarrows \C W \mmod$. Then Clifford's Theorem allows one to compare $\C K$ and $\C W$-modules via the induction and restriction functors, see \cite[Chapter 7]{CR} for details. When the quotient group is cyclic it is possible to deduce the following result (the proof of which can be found in \cite[Proposition 6.1]{7}). 

\begin{prop}
Fix $\lambda \in \LW$ and write $\Res_K^W \lambda = \mu_1 \oplus \ds \oplus \mu_k$, where each $\mu_i$ is nonzero and irreducible. Then
\begin{enumerate}
\item $C_{\mu_i} = (C^\vee_d / C_\lambda^\vee)^\vee \subset C_d$, hence $|C_{\mu_i}| \cdot |C^\vee_\lambda| = d$,
\item $C_d$ acts transitively on the set $\{ \mu_1, \ds , \mu_k \}$,
\item the $\mu_i$ are pairwise non-isomorphic,
\item $\Ind_K^W \mu_i = \bigoplus_{\delta \in C^\vee_d / C_\lambda^\vee} \delta \cdot \lambda$.
\end{enumerate}
\end{prop}

\subsection{} To relate the action of $C_d$ on $\HtK \mmod$ and $C_d^\vee$ on $\HW \mmod$ let us introduce the semisimple algebras 
\bdm
A_W := \HW / \rad \HW \qquad \textrm{ and } \qquad A_K := \HtK / \rad \HtK.
\edm
Note that $A_K \subset A_W$ and there are natural induction and restriction functors, $\Ind_{A_K}^{A_W}$ and $\Res_{A_K}^{A_W}$. The functors 
\bdm
E_W \, : \, \C W \mmod \rightarrow A_W \mmod, \qquad E_W(\lambda) := A_W \otimes_{\HW} \HW \otimes_{\C[\h^*]^{co W} \rtimes W} \lambda
\edm
\bdm
E_K \, : \, \C K \mmod \rightarrow A_K \mmod, \qquad E_K(\mu) := A_K \otimes_{\HtK} \HtK \otimes_{\C[\h^*]^{co W} \rtimes K} \mu
\edm
are equivalences of categories with $E_W(\lambda) = L(\lambda)$ and $E_K(\mu) = \tilde{L}(\mu)$ for $\lambda \in \LW$ and $\mu \in \LK$. 

\begin{lem}
The following diagram commutes up to natural equivalences.
\beq\label{eq:commute}
\vcenter{
\xymatrix{ \C W \mmod  \ar[r]^{ E_W } \ar[d]^{\Res_K^W} & A_W \mmod \ar[d]^{\Res_{A_K}^{A_W}} \\
\C K \mmod \ar[r]_{E_K} \ar[u]^{\Ind_K^W} & A_K \mmod \ar[u]^{\Ind_{A_K}^{A_W}} }
}
\eeq
\end{lem}

\begin{proof}
Let us write $\LW = \{ \lambda_1 , \ds, \lambda_k \}$, $\LK = \{ \mu_1, \ds, \mu_l \}$ and $a_{ij} \in \N$ such that $\Res_{K}^{W} \lambda_i = \oplus_j \, \mu_j^{\oplus a_{ij}}$. We begin by showing that the functors $E_W \circ \Ind_K^W$ and $\Ind_{A_K}^{A_W} \circ E_K$ are equivalent. The fact that $\C W = \bigoplus_{i} \lambda_i \otimes \lambda_i^*$ as a $\C W$-$\C W$-bimodule implies that $E_W (\C W)  =\bigoplus_{i} L(\lambda_i) \otimes \lambda_i^*$ as a $A_W$-$\C W$-bimodule. Similarly, $E_K (\C K) = \bigoplus_{j} \tilde{L}(\mu_j) \otimes \mu_j^*$ as a $A_K$-$\C K$-bimodule. Frobenius reciprocity implies that 
\bdm
E_W \circ \Ind_K^W \C K \simeq \bigoplus_{ij} L(\lambda_i) \otimes (\mu_j^*)^{\oplus a_{ij}}
\edm
as a $A_W$-$\C K$-bimodule. The isomorphism $\HW \otimes_{\HtK} \tilde{\Delta}(\mu_j) \simeq \Delta( \Ind_K^W \mu_j)$ implies that 
\bdm
\Ind_{A_K}^{A_W} \tilde{L}(\mu_j) \simeq \bigoplus_i L(\lambda_i)^{\oplus a_{ij}},
\edm
and thus
\bdm
\Ind_{A_K}^{A_W} \circ E_K (\C K) \simeq \bigoplus_{ij} L(\lambda_i) \otimes (\mu_j^*)^{\oplus a_{ij}},
\edm
as a $A_W$-$\C K$-bimodule. Since the functors $E_W \circ \Ind_K^W$ and $\Ind_{A_K}^{A_W} \circ E_K$ are exact, Watts' Theorem (\cite[Theorem 5.45]{Rotman}) says that $E_W \circ \Ind_K^W$ is naturally isomorphic to $E_W \circ \Ind_K^W (\C K) \otimes_{\C K} -$ and $\Ind_{A_K}^{A_W} \circ E_K$ is naturally isomorphic to $\Ind_{A_K}^{A_W} \circ E_K (\C K) \otimes_{\C K} -$. The required equivalence now follows from the general fact that if $A_1$ and $A_2$ are algebras, $B,C$ isomorphic $A_1$-$A_2$-bimodules then fixing an isomorphism $B \rightarrow C$ defines an equivalence 
\bdm
B \otimes_{A_2} - \, \stackrel{\sim}{\longrightarrow} C \otimes_{A_2} -  \, : \, A_1 \mmod \longrightarrow A_2 \mmod.
\edm
The fact that the functors $E_K \circ \Res_{A_K}^{A_W}$ and $\Res_{A_K}^{A_W} \circ E_W$ are equivalent follows from the facts that $E_W \circ \Ind_K^W$ and $\Ind_{A_K}^{A_W} \circ E_K$ are equivalent, $(\Ind_K^W, \Res_K^W)$ and $(\Ind_{A_K}^{A_W},\Res_{A_K}^{A_W})$ are pairs of adjoint functors and that $E_K$ and $E_W$ are equivalences of categories.
\end{proof}

\subsection{}\label{lem:equiv} The functors $E_W$ and $E_K$ behave well with respect to the groups $C_d^\vee$ and $C_d$. More precisely:

\begin{lem}
Let $\delta \in C_d^\vee$, $g \in C_d$, $\lambda \in \C W \mmod$ and $\mu \in \C K \mmod$, then  
\bdm
E_W(\delta \cdot \lambda) \simeq \delta \cdot E_W(\lambda) \quad \textrm{ and } \quad E_K(g \cdot \mu) \simeq g \cdot E_K(\mu).
\edm
\end{lem}

\begin{proof}
We prove that $E_W(\delta \cdot \lambda) = \delta \cdot E_W(\lambda)$, the argument for $E_K$ being similar. Consider the space $1 \otimes \lambda \otimes \delta \subset \delta \cdot \Delta(\lambda)$. For $\h \subset \C[\h^*]^{co W} \subset \HW$ we have $\h \cdot (1 \otimes \lambda \otimes \delta) = 0$, thus there is a nonzero map $\Delta(\delta \cdot \lambda) \rightarrow \delta \cdot \Delta(\lambda)$. The space $1 \otimes \lambda \otimes \delta$ generates $\delta \cdot \Delta(\lambda)$ therefore the map is an isomorphism. The head of $\Delta(\delta \cdot \lambda)$ is $E_W(\delta \cdot \lambda)$ and the head of $\delta \cdot \Delta(\lambda)$ is $\delta \cdot E_W(\lambda)$. This proves the result. 
\end{proof}

\subsection{}\label{prop:Clifford2} Combining Proposition \ref{prop:Clifford}, the commutativity of diagram (\ref{eq:commute}) and Lemma \ref{lem:equiv} we can conclude that

\begin{prop}
Fix $\lambda \in \LW$ and write $\Res_{A_K}^{A_W} L(\lambda) = \tilde{L}(\mu_1) \oplus \ds \oplus \tilde{L}(\mu_k)$, where each $\tilde{L}(\mu_i)$ is nonzero, irreducible. Then
\begin{enumerate}
\item $C_{\tilde{L}(\mu_i)} = C_{\mu_i}$ and $C_{L(\lambda)}^\vee = C_{\lambda}^\vee$.
\item $C_{\tilde{L}(\mu_i)} = (C^\vee_d / C_{L(\lambda)}^\vee)^\vee \subset C_d$, hence $|C_{\tilde{L}(\mu_i)}| \cdot |C^\vee_{L(\lambda)}| = d$,
\item $C_d$ acts transitively on the set $\{ \tilde{L}(\mu_1), \ds , \tilde{L}(\mu_k) \}$,
\item the $\tilde{L}(\mu_i)$ are pairwise non-isomorphic,
\item $\Ind_{A_K}^{A_W} \tilde{L}(\mu_i) = \bigoplus_{\delta \in C^\vee_d / C_{L(\lambda)}^\vee} \delta \cdot L(\lambda)$.
\end{enumerate}
\end{prop}

\subsection{}\label{lem:trivialaction} Since $C_d^\vee$ acts on the isomorphism classes of objects in $\HW \mmod$ and $C_d$ acts on the isomorphism classes of objects in $\HtK \mmod$, these groups also permute the blocks of the corresponding algebras. Hence there is an action of $C_d^\vee$ on the set $\CMW$ and an action of the group $C_d$ on the block partition of $\LK$ with respect to $\HtK$. 

\begin{lem}
The action of $C_d^\vee$ on $\CMW$ is trivial since each partition in $\CMW$ is a union of $C_d^\vee$ orbits. 
\end{lem}

\begin{proof}
 Let $\delta$ be a generator of $C_d^\vee$. Fix $B$ to be a block of $\HW$ and $\lambda \in \LW$ such that $L(\lambda)$ is a simple module for $B$. Then we must show that $L(\delta \cdot \lambda)$ is also a simple module for $B$. Since the baby Verma modules $\Delta(\lambda)$ and $\Delta(\delta \cdot \lambda)$ are indecomposable it suffices to show that there is a nonzero map $\Delta(\delta \cdot \lambda) \rightarrow \Delta(\lambda)$. In the notation of Lemma \ref{lem:algpoly}, $\C[U^*]^{co C_d}$ is isomorphic to the regular representation as a $C_d$-module. Let $\{ f_1, \ds, f_n \}$ be the set of generators described in Lemma \ref{lem:algpoly}. Then there exist $u_1, \ds, u_n$ with $0 \le u_i < a_i$ such that $g := f_1^{u_1} \cdots f_n^{u_n}$ equals $\delta$ as characters of $C_d$. Moreover the image of $g$ in $\C[\h]^{co W}$ is non-zero. The polynomial $g$ is $K$-invariant therefore it is central in $\HtK$. Since $\HtK \subset \HW$, $g$ commutes with the elements $\h \subset \HW$. Therefore the required map exists and is uniquely defined by $1 \otimes \delta \cdot \lambda \stackrel{\sim}{\longrightarrow} g \otimes \lambda$.
\end{proof}

\subsection{Twisted symmetric algebras} 
We shall show that $\HW$ is an example of a twisted symmetric algebra with respect to the group $C_d$. We follow the exposition given in \cite[Section $1$]{Chlou2} (see also \cite{Chlou3}). Although we do not use the properties of $\HW$ derived from the fact that it is a symmetric algebra we recall the relevant definitions for completeness. Let $A$ be a finite dimensional $\C$-algebra. 
\bdefn
A trace function on $A$ is a linear map $t \, : \, A \rightarrow \C$ such that $t(ab)= t(ba)$ for all $a, b \in A$. It is called a symmetrizing form on $A$, and $A$ itself is said to be a symmetric algebra, if the morphism
\bdm
\hat{t} \, : \, A \rightarrow \Hom_{\C}(A,\C), \qquad a \mapsto (\hat{t}(a) \, : \,  b \mapsto t(ab))
\edm
is an isomorphism of $(A,A)$-bimodules. 
\edefn

\begin{prop}[\cite{BGS}, Corollary 3.7]
The restricted rational Cherednik algebra $\HW$ is a symmetric algebra.
\end{prop}

\subsection{}\label{lem:subsymm} Let $A$ be a symmetric algebra with form $t$ and $B$ a subalgebra of $A$. Then $B$ is said to be a symmetric subalgebra of $A$ if the restriction of $t$ to $B$ is a symmetrizing form for $B$ and $A$ is free as a left $B$-module. 

\begin{lem}
The algebra $\HtK$ is a symmetric subalgebra of $\HW$.
\end{lem}

\begin{proof}
If $w_1, \ds , w_d$ are left coset representatives of $K$ in $W$, then the PBW property (\ref{eq:PBW}) implies that $\HW$ is a free left $\HtK$-module with basis $w_1, \ds , w_d$. The fact that the restriction of $t$ to $\HtK$ is symmetrizing is clear from the proof of \cite[Lemma 3.5]{BGS}.
\end{proof}

\begin{defn}
Following \cite[Definition 1.10]{Chlou2} we say that the symmetric algebra $(A,t)$ is a twisted symmetric algebra of a finite group $G$ over the subalgebra $B$ if $B$ is a symmetric subalgebra of $A$ and there is a family of vector subspaces $\{ A_g \, | g \in G \}$ of $A$ such that the following conditions hold:
\begin{enumerate}
\item $A = \bigoplus_{g \in G} A_g,$
\item $A_g A_h = A_{gh}$ for all $g,h \in G$,
\item $A_1 = B$,
\item $t(A_g) = 0$ for all $g \in G, \, g \neq 1$,
\item $A_g \cap A^{\times} \neq \emptyset$ for all $g \in G$ (here $A^{\times}$ are the units of $A$).
\end{enumerate}
\end{defn}

\begin{prop}
The symmetric algebra $\HW$ is a twisted symmetric group algebra of the group $C_d$ over the subalgebra $\HtK$.
\end{prop}

\begin{proof}
As in Lemma \ref{lem:subsymm}, let $w_1, \ds , w_d$ be left coset representatives of $K$ in $W$ and assume $C_d = \{ g_1, \ds , g_d\}$, such that $Kw_i = g_i$ in $W/K = C_d$. Then $\HW_{g_i} := \HtK \cdot w_i$. Conditions $(1), \, (3)$ and $(5)$ are clear. Since conjugation by $w_i$ defines an automorphism of $\HtK$, condition $(2)$ is also clear. Finally condition $(4)$ follows from the definition of the symmetrizing form $\Phi$ given in \cite[(3.5)]{BGS}.
\end{proof} 

\subsection{}\label{thm:compare1} We are now in a situation where we can apply \cite[Proposition 2.3.18]{Chlou3}. 

\begin{thm}
For $\mc{S} \subset \LW$, let $\Gamma(\mc{S})$ be the set of all $\mu \in \LK$ occurring as a summand of $\Res_{K}^{W} \, \lambda$ for some $\lambda \in \mc{S}$. Let $\mc{P} \in \CMW$. Then there exists $\mc{Q} \in \CMK$ such that $\Gamma(\mc{P}) = C_d \cdot \mc{Q}$. This implies that there is a bijection
\bdm
\CMW \stackrel{1 : 1}{\longleftrightarrow} \CMK / C_d.
\edm
\end{thm}

\begin{proof}
Proposition \ref{thm:lifting} tells us that $\{ \textrm{blocks of $\HtK$} \} = \CMK$. This identification is $C_d$-equivariant. Therefore it suffices to show that theorem holds but with $\CMK$ replaced by $\{ \textrm{blocks of $\HtK$} \}$. In \cite{Chlou3} Chlouveraki makes use of the existence of a field extension of the base field of the twisted symmetric algebra $A$ such that the extended symmetric algebra is split-semisimple. This fact is used to prove \cite[Proposition 2.3.15]{Chlou3}. Such an extension does not exist for $\HW$ but Proposition \ref{prop:Clifford2} is our substitute result. Now \cite[Proposition 2.3.18]{Chlou3} is applicable, with $A = \HW$ and $\bar{A} = \HtK$ since its proof does not explicitly rely on the existence of a ``splitting field extension''. This result says that the rule $C_d^\vee \cdot \mc{P} \mapsto \Gamma( C_d^\vee \cdot \mc{P})$ defines a bijection between the set of $C^\vee_d$-orbits in $\CMW$ and the $C_d$-orbits in $\{ \textrm{blocks of $\HtK$} \}$. However, Lemma \ref{lem:trivialaction} says that the action of $C_d^\vee$ on $\CMW$ is trivial.
\end{proof}

\subsection{}\label{lem:singletons} Let us note a particular situation where we can give a more precise result.

\begin{lem}
Let $\lambda \in \LW$ such that $\{ \lambda \} \in \CMW$. Then $\Res_K^W \lambda = \oplus_{i = 1}^d \mu_i$, $\mu_i \not\cong \mu_j$ for $i \neq j$ and $\{ \mu_i \} \in \CMK$ for $1 \le i \le d$.
\end{lem}

\begin{proof}
Again, since Proposition \ref{thm:lifting} tells us that $\{ \textrm{blocks of $\HtK$} \} = \CMK$ it suffice to show the statement holds with $\CMK$ replaced by $\{ \textrm{blocks of $\HtK$} \}$. Proposition \ref{prop:Clifford} tells us that $\Res_K^W \lambda = \oplus_{i = 1}^e \mu_i$ for some $e$ dividing $d$ and $\mu_i \not\cong \mu_j$ for $i \neq j$. Moreover, there exists $g \in C_d$ such that ${}^g \mu_i = \mu_j$ and hence ${}^g \tilde{L}(\mu_i) = \tilde{L}(\mu_j)$. In particular, $\dim \tilde{L}(\mu_i) = \dim \tilde{L}(\mu_j) = r$ for all $i,j$ and some $r \le |K|$. It is shown in \cite[(5.3)]{Baby} that $\dim L(\lambda) = |W|$ if and only if $\{ \lambda \}$ is a partition of $\CMW$. Proposition \ref{prop:Clifford2} says that $\Res_{A_K}^{A_W} L(\lambda) = \oplus_{i = 1}^e \tilde{L}(\mu_i)$. Comparing the dimension of both sides gives 
\bdm
|W| = e \cdot r \le d \cdot |K| = |W|.
\edm
Thus $e = d$ and $r = |K|$. Again, by \cite[(5.3)]{Baby}, $\dim \tilde{L}(\mu_i) = |K|$ implies that $\{ \mu_i \}$ is a block of $\HtK$. 
\end{proof}

\begin{remark}
In this article we focus on the particular case of $W = G(m,1,n)$ and $K = G(m,d,n)$ (details are given in section \ref{sec:example}). However, we believe that it is advantageous to present Theorem \ref{thm:compare1} in the level of generality that we have done here since there are many examples among the $34$ exceptional irreducible complex reflection groups of pairs $(W,K)$. Therefore in order to calculate the Calogero-Moser partition for all exceptional groups it would suffice to consider only certain groups. We refer the reader to the appendix of \cite{Chlou3} for a list of many such pairs $(W,K)$.
\end{remark}

\section{The imprimitive groups $G(m,d,n)$}\label{sec:example}

\subsection{} The irreducible complex reflection groups are divided into two classes, the primitive complex reflection groups and the imprimitive complex reflection groups. The groups were classified by Shephard and Todd in \cite{5}. There are 34 primitive complex reflection groups, which in the classification of \cite{5} are labelled $G_4, \dots , G_{37}$. They are also known as the exceptional complex reflection groups. In this section we will consider instead the imprimitive complex reflection groups. These belong to one infinite family $G(m,d,n)$, where $m,d,n \in \N$ and $d$ divides $m$. Let $S_n$ be the symmetric group on $n$ elements, considered as the group of all $n \times n$ permutation matrices. Let $A(m,d,n)$ be the group of all diagonal matrices whose diagonal entries are powers of a certain (fixed) $m^{th}$ root of unity and whose determinant is a $(m/d)^{th}$ root of unity. The group $S_n$ normalizes $A(m,d,n)$ and $G(m,d,n)$ is defined to be the semidirect product of $A(m,d,n)$ by $S_n$. Note that $G(m,1,n)$ is the wreath product group $C_m \wr S_n$. Fix $p = m/d$.

\subsection{The conjugacy classes of reflections}
Fix $\zeta$ a primitive $m^{th}$ root of unity. Let $s_{(i,j)} \in S_n$ denote the transposition swapping $i$ and $j$ and let $\varepsilon_i^k$ be the matrix in $A(m,1,n)$ which has ones all along the diagonal except in the $i^{th}$ position where its entry is $\zeta^k$. The conjugacy classes of reflections in $G(m,1,n)$ are 
\begin{displaymath}
R = \{ s_{(i,j)} \varepsilon_i^k \varepsilon_j^{-k} : 1 \le i \neq j \le n, 0 \le k \le m-1 \},
\end{displaymath}
\begin{displaymath}
S_i = \{ \varepsilon_j^i : 1 \le j \le n \}_{1 \le i \le m-1}.
\end{displaymath}  
The $G(m,1,n)$-conjugacy classes of reflections in $G(m,d,n)$ are
\begin{displaymath}
R = \{ s_{(i,j)} \varepsilon_i^k \varepsilon_j^{-k} : 1 \le i \neq j \le n, 0 \le k \le m-1 \},
\end{displaymath}
\begin{displaymath}
S_{id} = \{ \varepsilon_j^{id} : 1 \le j \le n \}_{1 \le i \le p-1}.
\end{displaymath}
The following is an application of \cite[Theorem 3]{16}. 

\begin{prop}
Let $n > 2$ or $n = 2$ and $d$ odd, then the $G(m,1,n)$-conjugacy classes of reflections in $G(m,d,n)$ coincide with the $G(m,d,n)$-conjugacy classes of reflections in $G(m,d,n)$. When $n = 2$ and $d$ is even the $G(m,d,2)$-conjugacy classes of reflections in $G(m,d,2)$ are
\begin{displaymath}
R_1 = \{ s_{(1,2)} \varepsilon_i^k \varepsilon_j^{-k} : 0 \le k \le m-1, k \textrm{ even} \}, \qquad R_2 = \{ s_{(1,2)} \varepsilon_i^k \varepsilon_j^{-k} : 0 \le k \le m-1, k \textrm{ odd} \},
\end{displaymath}
and 
\begin{displaymath}
S_{id} = \{ \varepsilon_j^{id} : 1 \le j \le n \}_{1 \le i \le p-1}.
\end{displaymath}
\end{prop}

\subsection{}\label{sec:unequal} The group $G(m,d,n)$ is a normal subgroup of $G(m,1,n)$ of index $d$ and the quotient group is the cyclic group $C_d$. Therefore we are in the situation considered in the previous sections. If $\mathbf{c}$ is a $G(m,d,n)$-conjugate invariant function on the set of reflections of that group then, provided $n \neq 2$ or $n = 2$ and $d$ is odd, $\mathbf{c}$ extends by zero to a $G(m,1,n)$-conjugate invariant function on the set of reflections of $G(m,1,n)$. If $n = 2$ and $d$ is even, we are restricted to considering $\mathbf{c}$ such that $\mathbf{c}(R_1) = \mathbf{c}(R_2)$. The group $C_d = \langle \varepsilon_1^p \rangle$ is a cyclic subgroup of $G(m,1,n)$ and normalises $G(m,d,n)$. If $d$ is co-prime to $p$ then $G(m,1,n) = G(m,d,n) \rtimes C_d$, an important example of this behaviour is $G(m,m,n) \triangleleft G(m,1,n)$. In such situations there exists an algebra isomorphism 
\begin{displaymath}
H_{t,\mathbf{c}}(G(m,1,n)) \cong H_{t,\mathbf{c}}(G(m,d,n)) \rtimes C_d.
\end{displaymath}
A specific example of this is $H_{t,(c,0)}(B_n) \cong H_{t,c}(D_n) \rtimes C_2$, where $B_n$ and $D_n$ are the Weyl groups of type $B$ and $D$ respectively (they correspond to $G(2,1,n)$ and $G(2,2,n)$).

\subsection{Representations of $G(m,d,n)$}\label{sec:reps} We begin by giving an explicit description of the simple $G(m,1,n)$-modules. This will allow us to give a combinatorial description of the action of the groups $C_d$ and $C_d^\vee$ as defined in (\ref{sec:Clifford}). Recall that a \textit{partition} of $n$ is a  sequence of positive integers $\lambda = (\lambda_1 \ge \lambda_2 \ge \ds \ge \lambda_k > 0)$ such that $n = | \lambda | := \sum_{i = 1}^k \lambda_k$. We call $k$ the \textit{length} of $\lambda$. The simple $S_n$-modules are parameterized by partitions of $n$. Let $V_{\lambda}$ denote the simple $S_n$-module labelled by the partition $\lambda$. The simple $C_m$-modules will be denoted $\C \cdot \omega_i$ (or simply $\omega_i$), $0 \le i < m$. If $C_m = \langle \varepsilon \rangle$ then $\varepsilon \cdot \omega_i = \zeta^i \omega_i$ (we may think of $C_m \subset G(m,1,n)$ such that $\varepsilon = \varepsilon_1$). Now let $U$ be any $C_m$-module and $V$ a $S_n$-module. The \textit{wreath product} $U \wr V$ is the $G(m,1,n)$-module, which as a vector space is $U^{\otimes n} \otimes V$ and whose module structure is uniquely defined by
\bdm
\varepsilon_i \cdot (u_1 \otimes \ds \otimes u_n \otimes v) = u_1 \otimes \ds \otimes \varepsilon \cdot u_i \otimes \ds \otimes u_n \otimes v,
\edm
and for $\sigma \in S_n$:
\bdm
\sigma \cdot (u_1 \otimes \ds \otimes u_n \otimes v) = u_{\sigma^{-1}(1)} \otimes \ds \otimes u_{\sigma^{-1}(n)} \otimes \sigma \cdot v.
\edm  
If $U$ and $V$ are simple modules then $U \wr V$ is a simple $G(m,1,n)$-module. However not every simple $G(m,1,n)$-module can be written in this way. A complete set of simple modules was originally constructed by Specht \cite{Specht}. The precise result is stated below, and a proof can be found in \cite[Theorem 4.3.34]{JK}. An $m$-multipartition $\underline{\lambda}$ of $n$ is an ordered $m$-tuple of partitions $(\lambda^0, \dots, \lambda^{m-1})$ such that $|\lambda^0| + \dots + |\lambda^{m-1}| = n$. Let $\mathcal{P}(m,n)$ denote the set of all $m$-multipartitions of $n$. To each $m$-tuple $n_0 + \ds + n_{m-1} = n$ there is a corresponding \textit{Young subgroup} $G_{(n)} = C_m \wr (S_{n_0} \times \ds \times S_{n_{m-1}})$ of $G(m,1,n)$. 

\begin{thm}\label{thm:Spetch}
To each $\underline{\lambda}$ in $\mathcal{P}(m,n)$ we can associate the $G(m,1,n)$-module
\bdm
V_{\underline{\lambda}} := \Ind_{G_{(n)}}^{G(m,1,n)} \, (\omega_0 \wr V_{\lambda^0}) \otimes \ds \otimes (\omega_{m-1} \wr V_{\lambda^{m-1}}),
\edm
where $G_{(n)}$ is the Young subgroup corresponding to the $m$-tuple $|\lambda^0| + \dots + |\lambda^{m-1}| = n$. Each $V_{\underline{\lambda}}$ is simple, $V_{\underline{\lambda}} \not\simeq V_{\underline{\mu}}$ for $\underline{\lambda} \neq \underline{\mu}$ and every simple $G(m,1,n)$-module is isomorphic to $V_{\underline{\lambda}}$ for some $\underline{\lambda}$.
\end{thm}

\subsection{} Note that in the case $n_i = 0$, the module $\omega_i \wr V_{\lambda^i}$ should be regarded as the one-dimensional trivial module. An element of $G(m,1,n)$ can be thought of as a permutation matrix but with the unique $1$ in each row replaced by an element of $C_m$. The rule that takes each such matrix to the product of its non-zero entries defines a character $\delta' \, : \, G(m,1,n) \rightarrow \C^*$ (this is not the determinant of the matrix). Fix $\delta := (\delta')^p$. Then $C_d^\vee = \langle \delta \rangle$ and it follows from (\ref{sec:Clifford}) that $(\omega_i \wr V) \otimes \delta \simeq \omega_{i + p} \wr V$. If we define the action of $C_d^\vee$ on $\underline{\lambda}$ by
\bdm
\delta \cdot (\lambda^0, \dots, \lambda^{m-1}) = (\lambda^{m - p}, \lambda^{m+1-p}, \dots ,\lambda^{m-2},\lambda^{m-1},\lambda^0,\lambda^1, \dots, \lambda^{m-p-1}),
\edm
then Theorem \ref{thm:Spetch} implies that $\delta \cdot V_{\underline{\lambda}} = V_{\delta \cdot \underline{\lambda}}$. We denote the orbit $C_d^\vee \cdot \underline{\lambda}$ by $\{ \underline{\lambda} \}$. Since $(C^\vee_d / C^\vee_{\underline{\lambda}})^\vee \subset C_d$ is the stabilizer $C_{\mu}$ of $\mu$, an irreducible summand of $\Res_{G(m,d,n)}^{G(m,1,n)} \, \underline{\lambda}$, we see by Proposition \ref{prop:Clifford} that the set of all irreducible summands of $\Res_{G(m,d,n)}^{G(m,1,n)}  \, \underline{\lambda}$ is parametrized by elements of the quotient $C_d / C_{\mu}$. This quotient can be identified with $C^\vee_{\underline{\lambda}}$ hence irreducible representations of $G(m,d,n)$ are parameterized by distinct pairs $(\{ \underline{\lambda} \},\epsilon)$, where $\epsilon \in C^\vee_{\underline{\lambda}}$. If we fix $C_d = \langle \, \overline{\varepsilon^p_1} \, \rangle$ and define the bijection $C_d \leftrightarrow C^\vee_d$ by $(\overline{\varepsilon^p_1})^i \leftrightarrow \delta^i$ then $C_d / C_\mu \leftrightarrow C^\vee_\lambda$ and the action of $C_d$ on pairs $(\{ \underline{\lambda} \},\epsilon)$ is given by 
\bdm
\eta \cdot (\{ \underline{\lambda} \},\epsilon) = (\{ \underline{\lambda} \},\eta \cdot \epsilon) \quad \textrm{ where } \quad (\eta \cdot \epsilon)( \nu ) = \epsilon (\eta \nu), \quad \textrm{ for } \eta, \nu \in C_d.
\edm

\section{Combinatorics} 

\subsection{} In this section we apply Theorem \ref{thm:compare1} to the combinatorial description of the partition $\mathsf{CM}_{\mbf{c}}(G(m,1,n))$ given in \cite{12} and deduce a similar description of the partition $\mathsf{CM}_{\mbf{c}}(G(m,d,n))$. First we must introduce some combinatorial objects.

\subsection{Young diagrams and $\beta$-numbers} Let $\lambda$ be a partition of $n$ of length $k$. The \textit{Young diagram} of $\lambda$ is defined to be the subset $Y(\lambda) := \{ (a,b) \in \Z^2 \, | \, 1 \le a \le k, \, 1 \le b \le \lambda_a \}$ of $\Z^2$. Each box in the diagram is called a \textit{node} and the \textit{content} of a node $(a,b)$ is defined to be the integer $\cont(a,b) := b - a$. The Young diagram should be visualized as a stack of boxes, justified to the left; for example the partition $(3,2,2,1)$ with its content is: 
\bdm
\Yboxdim18pt
\young(\minustwo :::,\minusone 01:,0123)
\edm

\subsection{Residues}\label{subsection:residue} Given a partition $\lambda$, we define the \textit{residue} of $\lambda$ to be the Laurent polynomial in $\Z [ x^{\pm 1}]$ given by
\bdm
\textrm{Res}_{\lambda}(x) := \sum_{(a,b) \in Y(\lambda)} x^{\cont (a,b)}.
\edm
For $r \in \Z$, the \textit{$r$-shifted} residue of $\lambda$ is defined to be $\textrm{Res}_{\lambda}^r(x) := x^r \textrm{Res}_{\lambda}(x)$. Let $\underline{\lambda} \in \mc{P}(m,n)$ and fix $\mbf{r} \in \Z^m$. Then the \textit{$\mbf{r}$-shifted residue} of $\underline{\lambda}$ is defined to be
\bdm
\textrm{Res}^{\mbf{r}}_{\underline{\lambda}} (x) := \sum_{i = 0}^{m-1} \textrm{Res}_{\lambda^i}^{r_i} (x).
\edm   

\subsection{}\label{subsection:translate} In order to use the combinatorics described in \cite{12} and \cite{Mo} we must change the basis of our parameter space. Recall that we have labelled the conjugacy classes of complex reflections in $G(m,1,n)$ as $R$ and $S_i$. We fix $\mathbf{c}(R) = k$ and $\mathbf{c}(S_i) = c_i$. The parameters of the rational Cherednik algebra $H_{\mbf{c}}(G(m,1,n))$ as used in \cite{12} are $\mbf{h} = (h,H_0, \dots , H_{m-1})$. We wish to find an expression for these parameters in terms of $k$ and $c_1, \dots ,c_{m-1}$. For the remainder of this section we make the assumption that $k \neq 0$. Without loss of generality $k = -1$. The parameter $H_0$ is chosen so that $H_0 + H_1 + \dots + H_{m-1} = 0$. Recall that $\zeta$ is a primitive $m^{th}$ root of unity. By \cite[(2.7)]{10} we know that $h = k$ and 
\begin{displaymath}
c_i = \sum_{j = 0}^{m-1}  \zeta^{-ij} H_j.
\end{displaymath}
Noting that 
\begin{displaymath}
\sum_{i = 1}^{m-1}  \zeta^{-i(r + j)} =  \left\{ \begin{array}{lcl}
m-1 & & \textrm{if } r + j \equiv 0 \, \mod \, m\\
-1 & & \textrm{otherwise},
\end{array} \right. 
\end{displaymath}
we have for $1 \le r \le m-1$:
\begin{displaymath}
\zeta^{-r} c_1 + \zeta^{-2r} c_2 + \dots + \zeta^{-(m-1)r} c_{m-1} = \sum_{i = 1}^{m-1} \zeta^{-ri} \sum_{j = 0}^{m-1} \zeta^{-ij} H_j \qquad 
\end{displaymath}
\begin{displaymath}
 \qquad = \sum_{j = 0}^{m-1} H_j \sum_{i = 1}^{m-1} \zeta^{-i(r + j)} = (m-1)H_{m - r} - \sum_{\substack{j = 0 \\ j \neq m - r}}^{m-1} H_j = mH_{m-r}.
\end{displaymath}
Thus for $1 \le r \le m-1$: 
\begin{displaymath}
H_r = \frac{1}{m} \sum_{i = 1}^{m-1}  \zeta^{-i(m - r)} c_i = \frac{1}{m} \sum_{i = 1}^{m-1}  \zeta^{ir} c_i.
\end{displaymath}

\subsection{The Calogero-Moser partition for $C_m \wr S_n$}\label{thm:Mo}
The results in \cite{12} and \cite{Mo} are only valid for rational values of $\mbf{h}$. Therefore, for the remainder of this chapter, we restrict to those parameters $\mbf{c}$ for $G(m,1,n)$ such that $\mbf{h} = (-1, H_0, H_1, \ds, H_{m-1}) \in \Q^{m+1}$. Choose $e \in \N$ such that $e H_i \in \Z$ for all $0 \le i \le m-1$ and fix 
\bdm
\mbf{s} = (0, eH_1, e H_1 + eH_2, \ds, e H_1 + \ds + e H_{m-1}) \in \Z^m.
\edm
Combining \cite[Theorem 2.5]{12} with the wonderful, but difficult combinatorial result \cite[Theorem 3.13]{Mo} gives:

\begin{thm}
The multipartitions $\underline{\lambda}, \underline{\mu} \in \mc{P}(m,n)$ belong to the same partition of $\mathsf{CM}_{\mbf{c}}(G(m,1,n))$ if and only if
\bdm
\textrm{Res}^{\mbf{s}}_{\underline{\lambda}}(x^e) = \textrm{Res}^{\mbf{s}}_{\underline{\mu}}(x^e).
\edm
\end{thm}

\subsection{}\label{lem:cyclicparameters} The $G(m,1,n)$-conjugacy classes of $G(m,d,n)$ are $R$ and $S_{id}$, where $1 \le i \le p - 1$. Thus a parameter $\mbf{c}$ for $G(m,1,n)$ is an extension by zero of a parameter for $G(m,d,n)$ if and only if $c_i = 0$ for all $i \not\equiv 0 \, \mod \, d$. Let us therefore assume that $c_i = 0$ for $i \not\equiv 0 \, \mod \, d$. 
\begin{lem}
We have $c_i = 0$ for all $i \not \equiv 0 \, \mod \, d$ if and only if $H_{i + p} = H_i$ for all $i$.
\end{lem}

\begin{proof}
First assume that $c_i = 0$ for all $i \not \equiv 0 \, \mod \, d$. Then
\begin{displaymath}
H_{i + p} = \frac{1}{m} \sum_{r = 1}^{p-1}  \zeta^{dr(i + p)} c_{dr} = \frac{1}{m} \sum_{r = 1}^{p-1}  \zeta^{dri} c_{dr} = H_i.
\end{displaymath}
Conversely, if $H_{i + p} = H_i$ for all $i$ then 
\begin{displaymath}
c_i = \sum_{j = 0}^{m-1}  \zeta^{-ij} H_j = \sum_{j = 0}^{p-1} H_j \sum_{r = 0}^{d-1} \zeta^{-i(j + rp)}.
\end{displaymath}
The result now follows from
\begin{displaymath}
\sum_{r = 0}^{d-1} \zeta^{-i(j + rp)} = \zeta^{-ij} \sum_{r = 0}^{d-1}  (\zeta^{-ip})^r =  \left\{ \begin{array}{lcl}
d \zeta^{-ij} & & \textrm{if } i \equiv 0 \, \mod \, d\\
0 & & \textrm{otherwise}.
\end{array} \right. 
\end{displaymath}
\end{proof}

\subsection{}\label{lem:Sdorbit} We will say that the parameter $\mbf{h} = (-1,H_0, \ds , H_{m-1})$ is \textit{$p$-cyclic} if $H_{i + p} = H_i$ for all $i$. Let $\underline{\lambda} = (\lambda^0, \dots , \lambda^{m-1})$ be an $m$-partition of $n$. We rewrite $\underline{\lambda}$ as $\underline{\lambda} = (\underline{\lambda}_0, \dots , \underline{\lambda}_{d-1})$ where $\underline{\lambda}_i = (\lambda^{ip}, \dots , \lambda^{(i+1)p - 1})$. Now the action of $C^\vee_d$ on $\underline{\lambda}$ as defined in (\ref{sec:reps}) can be expressed as 
\bdm
\delta \cdot (\underline{\lambda}_0, \ds , \underline{\lambda}_{d-1}) = (\underline{\lambda}_{d-1}, \underline{\lambda}_{0}, \ds , \underline{\lambda}_{d-2}).
\edm
An $m$-multipartition of $n$ is called \textit{$d$-stuttering} if $\underline{\lambda}_i = \underline{\lambda}_j$ for all $0 \le i,j \le d-1$. The group $C_d^\vee$ can be considered as a subgroup of $\mf{S}_d$, the symmetric group on $d$ elements, acting on $\mc{P}(m,n)$ as:
\bdm
\sigma \cdot (\underline{\lambda}_0, \ds , \underline{\lambda}_{d-1}) = (\underline{\lambda}_{\sigma(0)}, \ds , \underline{\lambda}_{\sigma(d-1)}).
\edm

\begin{lem}
Let $\mbf{c}$ be a parameter for $G(m,1,n)$ such that $\mbf{h} \in \Q^{m+1}$ is $p$-cyclic. Then the partitions of $\mathsf{CM}_{\mbf{c}}(G(m,1,n))$ consist of $\mf{S}_d$-orbits since
\bdm
\textrm{Res}^{\mbf{s}}_{\underline{\lambda}}(x^e) = \textrm{Res}^{\mbf{s}}_{\sigma \cdot \underline{\lambda}}(x^e),
\edm
where $\underline{\lambda} \in \mc{P}(m,n)$, $\sigma \in \mf{S}_d$ and $\mbf{s}$ is defined in (\ref{thm:Mo}).
\end{lem}

\begin{proof}
If $\mbf{h}$ is $p$-cyclic then the corresponding parameter $\mbf{s}$ has the form 
\bdm
\mbf{s} = (\mbf{s}' , \ds , \mbf{s}') \quad \textrm{ where } \quad \mbf{s}' = (0, eH_1, e H_1 + eH_2, \ds, e H_1 + \ds + e H_{p-1}),
\edm
and thus 
\bdm
\textrm{Res}^{\mbf{s}}_{\underline{\lambda}}(x^e) = \sum_{i = 0}^{d-1} \textrm{Res}^{\mbf{s}'}_{\underline{\lambda}_i}(x^e) \qquad \forall \, \underline{\lambda} \in \mc{P}(m,n) .
\edm  
Since the action of $\mf{S}_d$ simply reorders this sum, the result is clear.
\end{proof}

\subsection{}\label{lem:primedivisor} The following technical result will be needed later.

\begin{lem}
Let $\mbf{h}$ be a $p$-cyclic parameter and choose $\underline{\lambda} \in \mc{P}(m,n)$ to be a non $d$-stuttering $m$-multipartition of $n$. For each prime divisor $q$ of $d$, there exists an $m$-multipartition $\underline{\lambda}(q)$ of $n$ such that $\underline{\lambda}$ and $\underline{\lambda}(q)$ belong to the same partition of $\mathsf{CM}_{\mbf{c}}(G(m,1,n))$ and the order of the stabilizer of $\underline{\lambda}(q)$ under the action of $C^\vee_d$ is not divisible by $q$.
\end{lem}

\begin{proof}
We follow the argument given in \cite[Lemma 3.5]{14}. Since $\underline{\lambda}$ is not $d$-stuttering, there exists an $i > 0$ such that $\underline{\lambda}_i \neq \underline{\lambda}_0$. If $d = q$ there is nothing to prove so assume $d > q$ and set $l = d/q$, $l > 1$. Let $\sigma$ be the transposition in $\mf{S}_d$ that swaps $\underline{\lambda}_i$ and $\underline{\lambda}_{l-1}$ in $\underline{\lambda}$. We set $\underline{\lambda}(q) = \sigma \cdot \underline{\lambda}$. Then $\underline{\lambda}(q)$ is not fixed by any of the generators of the unique subgroup of $C^\vee_d$ of order $q$ and hence the stablizer subgroup of $\underline{\lambda}(q)$ has order co-prime to $q$. Since $\underline{\lambda}$ and $\underline{\lambda}(q)$ are in the same $\mf{S}_d$-orbit, Lemma \ref{lem:Sdorbit} says that they are in the same partition of $\mathsf{CM}_{\mbf{c}}(G(m,1,n))$. 
\end{proof}

\subsection{}\label{lem:dstuttering} We will also require the following result.

\begin{lem}
Let $\mbf{c}$ be a parameter for $G(m,1,n)$ such that $\mbf{h} \in \Q^{m+1}$ is $p$-cyclic and choose $\underline{\lambda} \in \mc{P}(m,n)$ to be $d$-stuttering. If $\{ \underline{\lambda} \}$ is not a partition of $\mathsf{CM}_{\mbf{c}}(G(m,1,n))$ then there exists a non $d$-stuttering $m$-multipartition $\underline{\mu}$ that is in the same partition as $\underline{\lambda}$. 
\end{lem}

\begin{proof}
Since $\{ \underline{\lambda} \}$ is not partition of $\mathsf{CM}_{\mbf{c}}(G(m,1,n))$ there must exist an $m$-multipartition $\underline{\lambda}' \neq \underline{\lambda}$ that is in the same partition as $\underline{\lambda}$. If $\underline{\lambda}'$ is not $d$-stuttering then we are done. Therefore we assume that $\underline{\lambda}'$ is $d$-stuttering. As noted in the proof of Lemma \ref{lem:Sdorbit}, $\mbf{h}$ being $p$-cyclic implies that 
\bdm
\textrm{Res}^{\mbf{s}}_{\underline{\mu}}(x^e) = \sum_{i = 0}^{d-1} \textrm{Res}^{\mbf{s}'}_{\underline{\mu}_i}(x^e) \qquad \forall \, \underline{\mu} \in \mc{P}(m,n) .
\edm
Hence $\textrm{Res}^{\mbf{s}}_{\underline{\lambda}}(x^e) = d \, \textrm{Res}^{\mbf{s}'}_{\underline{\lambda}_0}(x^e)$ and $\textrm{Res}^{\mbf{s}}_{\underline{\lambda}'}(x^e) = d \, \textrm{Res}^{\mbf{s}'}_{(\underline{\lambda}')_0}(x^e)$. It follows from Theorem \ref{thm:Mo} that 
\bdm
\textrm{Res}^{\mbf{s}'}_{\underline{\lambda}_0}(x^e) = \textrm{Res}^{\mbf{s}'}_{(\underline{\lambda}')_0}(x^e).
\edm
Set $\underline{\mu} = (\underline{\lambda}_0, (\underline{\lambda}')_0, \underline{\lambda}_0, \ds, \underline{\lambda}_0)$, it is a non $d$-stuttering $m$-multipartition. Again by Theorem \ref{thm:Mo}, $\textrm{Res}^{\mbf{s}}_{\underline{\lambda}}(x^e) = \textrm{Res}^{\mbf{s}}_{\underline{\mu}}(x^e)$ implies that $\underline{\lambda}$ and $\underline{\mu}$ belong to the same partition of $\mathsf{CM}_{\mbf{c}}(G(m,1,n))$. 
\end{proof}

\subsection{The main result}\label{thm:mainresult} Recall that for $\mc{P} \in \CMW$, $\Gamma(\mc{P})$ was defined to be the set of all $\mu \in \LK$ occurring as a summand of $\Res_{W}^{K} \, \lambda$ for each $\lambda \in \mc{P}$. In the case $W = G(m,1,n)$ and $K = G(m,d,n)$, $\Gamma$ is given combinatorially by $\Gamma(\mc{P}) = \{ \, ( \, \{ \underline{\lambda} \} \, , \epsilon ) \, | \, \underline{\lambda} \in \mc{P}, \, \epsilon \in C^\vee_{\underline{\lambda}} \}$. 

\begin{thm}
Let $\mbf{c} \, : \, \mc{S}(G(m,d,n)) \rightarrow \C$ be a $G(m,1,n)$-equivariant function such that $k \neq 0$ and $\mbf{h} \in \Q^{m+1}$. The $\mathsf{CM}_{\mbf{c}}(G(m,d,n))$ partition of $\textsf{Irr} \, (G(m,d,n))$ is described as follows. Let $\mc{Q}$ be a partition in $\mathsf{CM}_{\mbf{c}}(G(m,1,n))$:
\begin{enumerate}
\item If $\underline{\lambda}$ is a $d$-stuttering $m$-multipartition such that $\mc{Q} = \{ \underline{\lambda} \}$ then the sets $\{ ( \{ \underline{\lambda} \} , \epsilon ) \}$ where $\epsilon \in C^\vee_d$ are partitions of $\mathsf{CM}_{\mbf{c}}(G(m,d,n))$; 
\item Otherwise $\Gamma (\mc{Q})$ is a $\mathsf{CM}_{\mbf{c}}(G(m,d,n))$ partition of $\textsf{Irr} \, (G(m,d,n))$.
\end{enumerate}
\end{thm}

\begin{proof}
Rescaling if necessary, we may assume that $k = -1$. It is clear that the sets described in $(1)$ and $(2)$ of the theorem define a partition of the set $\textsf{Irr} \, (G(m,d,n))$. Therefore we just have to show that the sets describe the blocks of $\bar{H}_{\mbf{c}}(G(m,d,n))$. Proposition \ref{thm:lifting} says that it is sufficient to prove that $(1)$ and $(2)$ describe the equivalence classes of $\textsf{Irr} \, (G(m,d,n))$ with respect to the blocks of $\tilde{H}_{\mbf{c}}(G(m,d,n))$. Lemma \ref{lem:singletons} shows that the sets described in $(1)$ are indeed blocks of $\tilde{H}_{\mbf{c}}(G(m,d,n))$. So let us assume that $\mc{Q}$ is not of the form described in $(1)$. The group $C_d$ acts on the set $\Gamma(\mc{Q})$ and Theorem \ref{thm:compare1} says that there exists a block $B$ of $\tilde{H}_{\mbf{c}}(G(m,d,n))$ such that $C_d \cdot B = \Gamma(\mc{Q})$. We wish to show that $C_d \cdot B = B$. The fact that $g \cdot \tilde{L} \in g \cdot B$ for $\tilde{L} \in B$ and $g \in C_d$ implies that
\bdm
\bigcup_{\tilde{L} \in B} \textrm{Stab}_{C_d} \, \tilde{L}  \subseteq \textrm{Stab}_{C_d} \, B.
\edm
To show that $\textrm{Stab}_{C_d} \, B = C_d$ we will show that for every prime $q$ dividing $d$ there exists a $\tilde{L} \in B$ such that the highest power of $q$ dividing $d$ also divides $|\textrm{Stab}_{C_d} \, \tilde{L}(\mu)|$. This will imply $C_d \cdot B = B$ i.e. $\Gamma(\mc{Q}) = B$. Let $L(\lambda) \in \mc{Q}$ and let $\tilde{L}(\mu)$ be a summand of $\Res_{A_{G(m,d,n)}}^{A_{G(m,1,n)}} \, L(\lambda)$, then $\tilde{L}(\mu) \in g \cdot B$ for some $g \in C_d$. This means that $g^{-1} \cdot \tilde{L}(\mu) \in B$ is also a summand of $L(\lambda)$. Thus $\Res_{A_{G(m,d,n)}}^{A_{G(m,1,n)}} \, L(\lambda)$ contains a summand that lives in $B$, for all $L(\lambda) \in \mc{Q}$. Since $\textrm{Stab}_{C_d} \, \tilde{L}(\mu) = \textrm{Stab}_{C_d} \, \tilde{L}(\mu')$ for any two summands $\tilde{L}(\mu)$ and $\tilde{L}(\mu')$ of $\Res_{A_{G(m,d,n)}}^{A_{G(m,1,n)}} \, L(\lambda)$, it will suffice to show that, for every prime $q$ dividing $d$, there exists a $L(\lambda) \in \mc{Q}$ such that the highest power of $q$ dividing $d$ also divides $| \textrm{Stab}_{C_d} \, \tilde{L}(\mu) |$ for some summand $ \tilde{L}(\mu)$ of $\Res_{A_{G(m,d,n)}}^{A_{G(m,1,n)}} \, L(\lambda)$. Proposition \ref{prop:Clifford2} $(1)$ says that
\bdm
|\textrm{Stab}_{C_d} \, \tilde{L}(\mu)| \cdot | \textrm{Stab}_{C^\vee_d} \, L(\lambda)| = d.
\edm
Therefore it suffices to show that we can find $L(\lambda) \in \mc{Q}$ such that $q$ does not divide $| \textrm{Stab}_{C^\vee_d} \, L(\lambda)|$. Since $\mc{Q} \neq \{ \underline{\lambda} \}$ for some $d$-stuttering multipartition $\underline{\lambda}$, Lemma \ref{lem:dstuttering} says that there exists a non $d$-stuttering multipartition in $\mc{Q}$. Lemma \ref{lem:primedivisor} now says that the module $L(\lambda)$ we require exists in $\mc{Q}$.
\end{proof}

\begin{cor}
Let $\mbf{c} \, : \, \mc{S}(G(m,d,n)) \rightarrow \C$ be a $G(m,1,n)$-equivariant function such that $k = -1$ and $\mbf{h} \in \Q^{m+1}$, extended to a function $\mbf{c} \, : \,  \mc{S}(G(m,1,n)) \rightarrow \C$ and define $\mbf{s}$ as in (\ref{thm:Mo}). Choose $( \{ \underline{\lambda} \}, \epsilon) , ( \{ \underline{\mu} \}, \eta) \in \textsf{Irr} \, (G(m,d,n))$, then 
\begin{itemize}
\item if $\{ \underline{\lambda} \} \neq \{ \underline{\mu} \}$, then $( \{ \underline{\lambda} \}, \epsilon)$ and  $( \{ \underline{\mu} \}, \eta)$ are in the same partition of $\mathsf{CM}_{\mbf{c}}(G(m,d,n))$ if and only if
\bdm
\textrm{Res}^{\mbf{s}}_{\underline{\lambda}}(x^e) = \textrm{Res}^{\mbf{s}}_{\underline{\mu}}(x^e);
\edm
\item  if $\underline{\lambda} = \underline{\mu}$ is a $d$-stuttering partition and $\textrm{Res}^{\mbf{s}}_{\underline{\lambda}}(x^e) \neq \textrm{Res}^{\mbf{s}}_{\underline{\nu}}(x^e)$ for all $\underline{\lambda} \neq \underline{\nu} \in \mc{P}(m,n)$ then $( \{ \underline{\lambda} \}, \epsilon)$ and  $( \{ \underline{\lambda} \}, \eta)$ are in the same partition of $\mathsf{CM}_{\mbf{c}}(G(m,d,n))$ if and only if $\epsilon = \eta$; 
\item otherwise $( \{ \underline{\lambda} \}, \epsilon)$ and  $( \{ \underline{\lambda} \}, \eta)$ are in the same partition of $\mathsf{CM}_{\mbf{c}}(G(m,d,n))$.
\end{itemize}
\end{cor}

\subsection{}\label{lem:generic}
It was shown by the author in \cite{Singular} that the partition $\mathsf{CM}_{\mbf{c}}(G(m,d,n))$ is never trivial, even for generic values of $\mbf{c}$. Here we describe $\mathsf{CM}_{\mbf{c}}(G(m,d,n))$ for generic $\mbf{c}$.

\begin{lem}
Let $\mbf{c}$ be a generic parameter for $H_{\mbf{c}}(G(m,d,n))$ such that $k \neq 0$ and $\mbf{h} \in \Q^{m+1}$. Choose $( \{ \underline{\lambda} \}, \epsilon) , ( \{ \underline{\mu} \}, \eta) \in \textsf{Irr} \, (G(m,d,n))$,
\begin{itemize}
\item  if $\underline{\lambda}$ is a $d$-stuttering partition then $\{ \, ( \{ \underline{\lambda} \}, \epsilon) \, \}$ is a partition of $\mathsf{CM}_{\mbf{c}}(G(m,d,n))$. 
\item otherwise $( \{ \underline{\lambda} \}, \epsilon)$ and  $( \{ \underline{\mu} \}, \eta)$ are in the same partition of $\mathsf{CM}_{\mbf{c}}(G(m,d,n))$ if and only if 
\beq\label{eq:sum}
\sum_{i = 0}^{d-1} \textrm{Res} {}_{\lambda^{j + pi}}(x^e) = \sum_{i = 0}^{d-1} \textrm{Res} {}_{\mu^{j + pi}}(x^e) \quad \forall \, 0 \le j \le p-1.
\eeq
\end{itemize}
Note that the expressions in (\ref{eq:sum}) are independent of the choice of representatives $\underline{\lambda} \in \{ \underline{\lambda} \}$ and $\underline{\mu} \in \{ \underline{\mu} \}$.
\end{lem}

\begin{proof}
Since $\mbf{h}$ is cyclic, we note once again that the vector $\mbf{s}$ as defined in (\ref{thm:Mo}) has the form
\bdm
\mbf{s} = (\mbf{s}' , \ds , \mbf{s}') \quad \textrm{ where } \quad \mbf{s}' = (0, eH_1, e H_1 + eH_2, \ds, e H_1 + \ds + e H_{p-1}).
\edm
Therefore 
\bdm
\textrm{Res}^{\mbf{s}} {}_{\underline{\lambda}}(x^e) = \sum_{j = 0}^{p-1} x^{e \mbf{s}_j} \left( \sum_{i = 0}^{d-1} \textrm{Res} {}_{\lambda^{j + pi}}(x^e) \right),
\edm
and thus the genericity of $\mbf{c}$ implies that 
\bdm
\textrm{Res}^{\mbf{s}} {}_{\underline{\lambda}}(x^e) = \textrm{Res}^{\mbf{s}} {}_{\underline{\mu}}(x^e) \Leftrightarrow \sum_{i = 0}^{d-1} \textrm{Res} {}_{\lambda^{j + pi}}(x^e) = \sum_{i = 0}^{d-1} \textrm{Res} {}_{\mu^{j + pi}}(x^e) \quad \forall \, 0 \le j \le p-1.
\edm
If $\underline{\lambda}$ is $d$-stuttering then $\sum_{i = 0}^{d-1} \textrm{Res} {}_{\lambda^{j + pi}}(x^e) = d \, \textrm{Res} {}_{\lambda^{j}}(x^e)$, $\forall \, 0 \le j \le p-1$. It can easily be shown that if
\bdm
d \, \textrm{Res} {}_{\lambda^{j}}(x^e) = \sum_{i = 0}^{d-1} \textrm{Res} {}_{\mu^{j + pi}}(x^e)
\edm
then $\mu^{j + pi} = \lambda^j$ for all $i$. Therefore each $d$-stuttering partition forms a singleton partition in $\mathsf{CM}_{\mbf{c}}(G(m,1,n))$. Now the Lemma follows from Corollary \ref{thm:mainresult}.
\end{proof}

\section{Relation to Rouquier families}

\subsection{Generic Hecke algebras}In this section we show that Theorem \ref{thm:mainresult} confirms Martino's conjecture when $W = G(m,d,n)$. To each complex reflection group it is possible to associate a generic Hecke algebra. We recall the definition as given in \cite{Mo} (see also \cite{BMR}). Denote by $\mc{K}$ the set of all hyperplanes in $\h$ that are the fixed point sets of the complex reflections in $W$. The group $W$ acts on $\mc{K}$. Given $H \in \mc{K}$, the parabolic subgroup of $W$ that fixes $H$ pointwise is a rank one complex reflection group and thus isomorphic to the cyclic group $C_e$ for some $e$. Therefore an orbit of hyperplanes $\mc{C} \in \mc{K}$ corresponds to a conjugacy class of rank one parabolic subgroups, all isomorphic to $C_{e_{\mc{C}}}$. Let $e_{\mc{C}} := |C_{e_{\mc{C}}}|$ be the order of these parabolic subgroups. For every $d > 1$, fix $\eta_d = e^{\frac{2 \pi i}{d}}$ and let $\mu_{d}$ be the group of all $d^{th}$ roots of unity in $\C$. If $\mu_{\infty}$ is the group of all roots of unity in $\C$ then we choose $K$ to be some finite field extension of $\Q$ contained in $\Q (\mu_{\infty})$ such that $K$ contains $\mu_{e_{\mc{C}}}$ for all $\mc{C} \in \mc{K} / W$. The group of roots of unity in $K$ is denoted $\mu(K)$ and the ring of integers in $K$ is $\Z_K$.

\subsection{} Fix a point $x_0 \in \h_{\textrm{reg}} := \h \backslash \bigcup_{H \in \mc{K}} H$ and denote by $\bar{x}_0$ its image in $\h_{\textrm{reg}} / W$. Let $B$ denote the fundamental group $\Pi_1 (\h_{\textrm{reg}} / W , \bar{x}_0)$. Let $\mbf{u} = \{ \, (u_{\mc{C},j}) \, : \mc{C} \in \mc{K} / W, \, 0 \le j \le e_{\mc{C}} - 1 \}$ be a set of indeterminates, and denote by $\Z[\mbf{u},\mbf{u}^{-1}]$ the ring $\Z[u_{\mc{C},j}^{\pm 1} \, : \mc{C} \in \mc{K} / W, \, 0 \le j \le e_{\mc{C}}-1]$. The \textit{generic Hecke algebra}, $\mc{H}_W$, is the quotient of $\Z[ \mbf{u}, \mbf{u}^{-1}]B$ by the relations of the form
\bdm
(\mbf{s} - u_{\mc{C},0})(\mbf{s} - u_{\mc{C},1}) \cdots (\mbf{s} - u_{\mc{C},e_{\mc{C}} - 1}), 
\edm
where $\mc{C} \in \mc{K}/W$ and $\mbf{s}$ runs over the set of monodromy generators around the images in $\h_{\textrm{reg}} / W$ of the hyperplane orbit $\mc{C}$. The following properties are known to hold for all but finitely many complex reflection groups (it is conjectured that they hold for all complex reflection groups). In particular, they hold for the infinite series $G(m,d,n)$.
\begin{itemize}
\item $\mc{H}_W$ is a free $\Z[\mbf{u}, \mbf{u}^{-1}]$-module of rank $|W|$.
\item $\mc{H}_W$ has a symmetrizing form $t \, : \, \mc{H}_W \rightarrow \Z[\mbf{u}, \mbf{u}^{-1}]$ that coincides with the standard symmetrizing form on $\Z_K W$ after specializing $u_{\mc{C},j}$ to $\eta^j_{e_\mc{C}}$. 
\item Let $\mbf{v} = \{ (v_{\mc{C},j}) \, : \mc{C} \in \mc{K} / W, \, 0 \le j \le e_{\mc{C}}-1 \}$ be a set of indeterminates such that $u_{\mc{C},j} = \eta^j_{e_{\mc{C}}} v_{\mc{C},j}^{| \mu(K)|}$. Then the $K(\mbf{v})$-algebra $K(\mbf{v})\mc{H}_W$ is split semisimple.
\end{itemize}
Note that Tits' deformation theorem, \cite[Theorem 7.2]{GeckPff}, implies that the specialization $v_{\mc{C},j} \mapsto 1$ induces a bijection $\LW \leftrightarrow \textsf{Irr} \, K(\mbf{v})\mc{H}_W$.

\begin{remark} When $W = G(m,1,n)$ the set $\mc{K}/W$ is $\{ \mc{R}, \mc{S} \}$ where $\mc{R}$ is the orbit of hyperplanes that define the reflections in the conjugacy class $R$ and $\mc{S}$ is the orbit of hyperplanes defining the reflections in the conjugacy classes $S_0, \ds, S_{m-1}$. Therefore $e_{\mc{R}} = 2$ and $e_{\mc{S}} = m$. Similarly, when $W = G(m,d,n)$ and $n \neq 2$ or $n = 2$ and $p$ odd the set $\mc{K}/W$ is $\{ \mc{R}, \mc{S} \}$ where $\mc{R}$ is the orbit of hyperplanes that define the reflections in the conjugacy class $R$ and $\mc{S}$ is the orbit of hyperplanes defining the reflections in the conjugacy classes $S_d, \ds, S_{d(p-1)}$. Therefore $e_{\mc{R}} = 2$ and $e_{\mc{S}} = p$. However, when $W = G(m,d,2)$ with $d$ even, the set $\mc{K}/W$ is $\{ \mc{R}_1, \mc{R}_2, \mc{S} \}$, where $\mc{R}_1$, $\mc{R}_2$ are the orbits of the hyperplanes that define the reflections in the conjugacy classes $R_1$ and $R_2$. Here $e_{\mc{R}_1} = e_{\mc{R}_2} = 2$ and $e_{\mc{S}} = p$. 
\end{remark}

\subsection{Cyclotomic Hecke algebras}
The cyclotomic Hecke algebras are certain specializations of the generic Hecke algebra. Let $y$ be an indeterminate.
\begin{defn}
A cyclotomic Hecke algebra is the $\Z_K[y,y^{-1}]$-algebra induced from $\Z[\mbf{v},\mbf{v}^{-1}] \mc{H}_W$ by an algebra homomorphism of the form 
\bdm
\Z_K [\mbf{v},\mbf{v}^{-1}] \rightarrow \Z_K[y,y^{-1}], \qquad v_{\mc{C},j} \mapsto y^{n_{\mc{C},j}},
\edm
where the tuple $\mbf{n} := \{ (n_{\mc{C},j} \in \Z)  \, : \mc{C} \in \mc{K} / W, \, 0 \le j \le e_{\mc{C}}-1 \}$ is chosen such that the following property holds. Set $ x:= y^{|\mu (K) |}$ and let $z$ be an indeterminate. Then the element of $\Z_K[y,z]$ defined by
\bdm
\Gamma_{\mc{C}}(y,z) = \prod^{e_{\mc{C}} - 1}_{j = 0} (z - \eta^j_{e_{\mc{C}}} y^{n_{\mc{C},j}})
\edm
is required to be invariant under $\textrm{Gal} \, (K(y)/ K(x))$ for all $\mc{C} \in \mc{K}/W$. In other words, $\Gamma_{\mc{C}}(y,z)$ is contained in $\Z_K[x^{\pm 1},z]$. The cyclotomic Hecke algebra corresponding to $\mbf{n}$ is denoted $\mc{H}_W(\mbf{n})$.
\end{defn}
The symmetric form $t$ on $\mc{H}_W$ induces a symmetrizing form on $K(y)\mc{H}_W(\mbf{n})$ and this algebra is split semisimple by \cite[(4.3)]{Chlou3}. Therefore Tits' deformation theorem implies that we have bijections
\bdm
\LW \leftrightarrow \textsf{Irr} \, K(y)\mc{H}_W(\mbf{n}) \leftrightarrow K(\mbf{v})\mc{H}_W.
\edm

\subsection{Rouquier families}
The \textit{Rouquier ring} is defined to be $\mc{R}(y) = \Z_K[y,y^{-1},(y^n - 1)^{-1} \, : \, n \in \N]$. Since $\mc{H}_W$ is free of rank $|W|$, $\mc{R}(y)\mc{H}_W(\mbf{n}) \subset K(y)\mc{H}_W(\mbf{n})$ is also free of rank $|W|$. We define an equivalence relation on $\textsf{Irr} \, K(y)\mc{H}_W(\mbf{n}) = \LW$ by  saying that $\lambda \sim \mu$ if and only if $\lambda$ and $\mu$ belong to the same block of $\mc{R}(y)\mc{H}_W(\mbf{n})$. The equivalence classes of this relation are called \textit{Rouquier families}. 

\subsection{}\label{sec:fixparametersconjecture} Fix a parameter $\mbf{c}$ for $G(m,d,n)$ that extends to a parameter $\mbf{c}$ for $G(m,1,n)$, translated into the form $\mbf{h} = (h,H_0, \ds , H_{m-1})$ as described in (\ref{subsection:translate}). Again we make the assumption that $h = -1$ and $\mbf{h} \in \Q^{m+1}$. Choose $e \in \N$ such that $eh$ and $eH_i \in \Z$ for all $0 \le i \le m-1$. Then $\mbf{n} = (n_{\mc{R},0},n_{\mc{R},1},n_{\mc{S},0}, \ds , n_{\mc{S},m-1})$ is fixed to be $n_{\mc{R},0} = e, n_{\mc{R},1} = 0$ and $n_{\mc{S},j} = e \sum_{i = 1}^j H_i$ for $0 \le j \le m - 1$. From now on we fix $K = \Q (\eta_m)$ and $\Z_K = \Z[\eta_m]$. Recall the morphism $\Upsilon$ defined in (\ref{sec:restricteddefinition}).

\begin{conjecture}[Martino, \cite{Mo}, (2.7)]
Let $\mbf{c},\mbf{h}$ and $\mbf{n}$ be as above.
\begin{enumerate}
\item The partition of $\textsf{Irr} \, G(m,d,n)$ into Rouquier families associated to $\mc{H}_{G(m,d,n)}(\mbf{n})$ refines the\\ $\textsf{CM}_{\mbf{c}}(G(m,d,n))$ partition. For generic values of $\mbf{c}$ the partitions are equal.
\item Let $q \in \Upsilon^{-1}(0)$ and let $K(y)B_1 \oplus \cdots \oplus K(y)B_k$ be the sum of the corresponding Rouquier blocks. Then $\dim \, (\C[ \Upsilon^{*}(0)_q]) = \dim_{K(y)} \, K(y)B_1 \oplus \cdots \oplus K(y)B_k$.
\end{enumerate}
\end{conjecture}

\subsection{}\label{sec:essential} The Rouquier families for $G(m,1,n)$ are calculated by Chlouveraki \cite{Chlou1} using the idea of \textit{essential hyperplanes}. The essential hyperplanes for $G(m,1,n)$ in $\Z^{m+1}$ are of the form $(k n_{\mc{R},0} + n_{\mc{S},i} - n_{\mc{S},j} = 0)$ for $0 \le i < j \le m-1$ and $-m < k < m$, and $(n_{\mc{R},0} = 0)$. 

\begin{defn}
Let $\mbf{n} \in \Z^{m+1}$. 
\begin{itemize}
\item The hyperplane $(k n_{\mc{R},0} + n_{\mc{S},i} - n_{\mc{S},j} = 0)$ is said to be \textit{essential} if there exists a prime ideal $\mf{p}$ of $\Z [\eta_m]$ such that $\eta_m^i - \eta_m^j \in \mf{p}$. The hyperplane $(n_{\mc{R},0} = 0)$ is always assumed to be essential.
\item If $\mbf{n}$ does not belong to any essential hyperplane then $\mbf{n}$ is said to be \textit{generic}.
\item If $\mbf{n}$ belongs to the essential hyperplane $(k n_{\mc{R},0} + n_{\mc{S},i} - n_{\mc{S},j} = 0)$ and $\mbf{n}$ does not belong to any other essential hyperplane then $\mbf{n}$ is said to be a \textit{generic} element of $(k n_{\mc{R},0} + n_{\mc{S},i} - n_{\mc{S},j} = 0)$. 
\end{itemize}
\end{defn}

\subsection{}\label{prop:hyperpartition} If $\mbf{n} \in \Z^{m+1}$ does not belong to any essential hyperplane then the corresponding Rouquier families are independent of the choice of $\mbf{n}$. Similarly, if $\mbf{n}$ is a generic element in some essential hyperplane then the Rouquier families for $\mbf{n}$ are independent of the choice of $\mbf{n}$. A general element $\mbf{n} \in \Z^{m+1}$ will belong to a collection of essential hyperplanes $H_1, \ds, H_k = 0$. It has been shown by Chlouveraki \cite{Chlou3} that Rouquier families have the property of \textit{semi-continuity}. This means that the partition of $\textsf{Irr} \, G(m,1,n)$ into Rouquier families for $\mbf{n}$ is the finest partition of $\textsf{Irr} \, G(m,1,n)$ that is refined by the Rouquier families partition of $\textsf{Irr} \, G(m,1,n)$ associated to each of the essential hyperplanes $H_i = 0$. Therefore if $\underline{\lambda}$ and $\underline{\mu}$ are in the same Rouquier family for some essential hyperplane $H_i = 0$ then they are in the same Rouquier family for $\mbf{n}$. 

\begin{prop}[\cite{Chlou1}, Proposition 3.15]
Let $(n_{\mc{S},i} - n_{\mc{S},j} = 0)$ be an essential hyperplane and choose $\mbf{n}$ to be a generic element of $(n_{\mc{S},i} - n_{\mc{S},j} = 0)$. Then $\underline{\lambda}, \underline{\mu} \in \mc{P}(n,m)$ are in the same Rouquier family of $\mc{R}(y)\mc{H}_{G(m,1,n)}(\mbf{n})$ if and only if
\begin{enumerate}
\item $\lambda^{a} = \mu^a$ for all $a \neq s,t$; and
\item $\textrm{Res} \, {}_{(\lambda^s,\lambda^t)}(x) = \textrm{Res} \, {}_{(\mu^s,\mu^t)}(x)$.
\end{enumerate}
\end{prop}

\begin{proof}
The result \cite[Proposition 3.15]{Chlou1} is stated in terms of weighted content but \cite[Proposition 3.4]{BroueKim} shows that we can reformulate the result in terms of residues. The weighting is $(0,k)$, which in our case becomes $(0,0)$ since $k = 0$.
\end{proof}

\begin{lem}
Let $\underline{\lambda}, \underline{\mu} \in \mc{P}(m,n)$. We write $\underline{\lambda} \sim \underline{\mu}$ if there exists $0 \le i \le p-1$ and $0 \le j < k \le d-1$ such that $\lambda^a = \mu^a$ for all $a \neq i + jp, i + kp$ and
\bdm
\textrm{Res} \, {}_{(\lambda^{i + jp},\lambda^{i + kp})}(x) = \textrm{Res} \, {}_{(\mu^{i + jp},\mu^{i + kp})}(x).
\edm
Now choose $\mbf{n}$ to be a generic parameter for $\mc{H}_{G(m,d,n)}$. Then the partition of $\textsf{Irr} \, G(m,1,n)$ into Rouquier families for $\mc{R}(y)\mc{H}_{G(m,1,n)}(\mbf{n})$ is the set of equivalence classes in $\textsf{Irr} \, G(m,1,n)$ under the transitive closure of $\sim$.
\end{lem}

\begin{proof}
Since $\mbf{n}$ is generic for $\mc{H}_{G(m,d,n)}$, the parameter $\mbf{h}$ satisfies $H_{i + p} = H_i$ for all $i$ and no other linear relations. Therefore it follows from (\ref{sec:fixparametersconjecture}) that the only hyperplanes that might be essential for $\mbf{n}$ (now considered a parameter for $\mc{H}_{G(m,1,n)}$) are of the form $(n_{\mc{S},i +jp} - n_{\mc{S},i + kp} = 0)$ for $0 \le i \le p-1$ and $0 \le j < k \le d-1$. However not all of these hyperplanes will be essential. Let us say that the $m$-multi-partition $\underline{\lambda}$ is \textit{linked} to the $m$-multi-partition $\underline{\mu}$ if there exists an essential hyperplane $(n_{\mc{S},i +jp} - n_{\mc{S},i + kp} = 0)$ containing $\mbf{n}$ such that 
$$
\textrm{Res} \, {}_{(\lambda^{i + jp},\lambda^{i + kp})}(x) = \textrm{Res} \, {}_{(\mu^{i + jp},\mu^{i + kp})}(x).
$$
Then, by Proposition \ref{prop:hyperpartition} and the principal of semi-continuity, the Rouquier families for $\mc{R}(y)\mc{H}_{G(m,1,n)}(\mbf{n})$ are the set of equivalence classes in $\textsf{Irr} \, G(m,1,n)$ under the transitive closure of ``linked''. Since $\underline{\lambda}$ linked $\underline{\mu}$ implies that $\underline{\lambda} \sim \underline{\mu}$, the Rouquier families refine the partition defined by $\sim$. Therefore we must show that if $\underline{\lambda} \sim \underline{\mu}$ (via $i + jp, i+ kp$ say) then there exists a chain of $m$-multi-partitions $\underline{\lambda} = \underline{\lambda}_1, \ds, \underline{\lambda}_q =  \underline{\mu}$ such that $\underline{\lambda}_{\alpha}$ is linked to $\underline{\lambda}_{\alpha+1}$ for all $1 \le \alpha \le q-1$. For each $0 \le i \le p-1$ and $0 \le j \le d-1$, the result \cite[Lemma 3.6]{Chlou2} says that the multi-partitions $\underline{\lambda}$ and $(i, i+ j p) \cdot \underline{\lambda}$ belong to the same Rouquier family for $\mc{R}(y)\mc{H}_{G(m,1,n)}(\mbf{n})$, where $(i, i + jp)$ is the transposition swapping the partitions $\lambda^{i}$ and $\lambda^{i + j p}$. In particular, this result (assuming that $d > 1$) shows that there exists some $l \neq 0$ such that the hyperplane $(n_{\mc{S},i} - n_{\mc{S},i + l p } = 0)$ is essential. Applying the result \cite[Lemma 3.6]{Chlou2}, we see that $\underline{\lambda}$ is in the same Rouquier family as 
$$
\underline{\lambda}'  := (i, i+ kp) \cdot (i + l p, i + j p) \cdot \underline{\lambda}
$$ 
and $\underline{\mu}$ is in the same Rouquier family as
$$
\underline{\mu}'  := (i, i+ kp) \cdot (i + l p, i + j p) \cdot \underline{\mu}.
$$ 
Now $(\lambda')^a = (\mu')^a$ for all $a \neq i, i + lp$ and 
\bdm
\textrm{Res} \, {}_{((\lambda')^{i},\lambda^{i + lp})}(x) = \textrm{Res} \, {}_{((\mu')^{i},\mu^{i + lp})}(x).
\edm
Since the hyperplane $(n_{\mc{S},i} - n_{\mc{S},i + l p } = 0)$ is essential, this implies that $\underline{\lambda}'$ is linked to $\underline{\mu}'$ and there must be a chain from $\underline{\lambda}$ to $\underline{\mu}$ as required.
\end{proof}

\subsection{}\label{sec:residueHecke} We will require the following combinatorial result. The proof uses the representation theory of cyclotomic Hecke algebras, it would be interesting to have a direct combinatorial proof.   

\begin{lem}
Let $\underline{\lambda}$ and $\underline{\mu}$ be two $m$-multi-partitions of $n$. Then $\textrm{Res}_{\underline{\lambda}}(x) = \textrm{Res}_{\underline{\mu}}(x)$ if and only if there exists a sequence of multipartitions $\underline{\lambda} = \underline{\lambda}(1), \ds, \underline{\lambda}(k) = \underline{\mu} \in \mc{P}(m,n)$ and $s(i) \neq t(i) \in \{1, \ds, m \}$, $1 < i \le k$, such that 
\begin{enumerate}
\item $\lambda(i-1)^a = \lambda(i)^a$ for all $a \neq s(i), t(i)$; and 
\item $\textrm{Res} \, {}_{(\lambda(i-1)^{s(i)},\lambda(i-1)^{t(i)})}(x) = \textrm{Res} \, {}_{(\lambda(i-1)^{s(i)},\lambda(i-1)^{t(i)})}(x), \quad \forall \, 1 < i \le k$.
\end{enumerate} 
\end{lem}

\begin{proof}
Let us fix $\mbf{n} = (n_{\mc{R},0},n_{\mc{R},1},n_{\mc{S},0}, \ds , n_{\mc{S},m-1})$ with $n_{\mc{R},0} = 1$ ,$n_{\mc{R},1} = 0$ and $n_{\mc{S},i} = 0$ for all $0 \le i \le m-1$. Then the Lemma is the result \cite[Proposition 3.19]{Chlou1} for our special parameter $\mbf{n}$, noting once again that \cite[Proposition 3.4]{BroueKim} allows us to rephrase \cite[Proposition 3.19]{Chlou1}, which is stated in terms of weighted content, in language of residues.
\end{proof}

\subsection{} We can now confirm the first part of Martino's conjecture for $G(m,d,n)$.

\begin{thm}
Let $\mbf{c} \, : \, \mc{S}(G(m,d,n)) \rightarrow \C$ be a $G(m,1,n)$-equivariant function such that $k = -1$ and $\mbf{h} \in \Q^{m+1}$. Choose $e \in \N$ such that $eh$ and $eH_i \in \Z$ for all $0 \le i \le m-1$. Fix $n_{\mc{R},0} = e, n_{\mc{R},1} = 0$ and $n_{\mc{S},j} = e \sum_{i = 1}^j H_i$ for $0 \le j \le m - 1$. Then
\begin{enumerate}
\item the partition of $\textsf{Irr} \, G(m,d,n)$ into Rouquier families associated to $\mc{H}_{G(m,d,n)}(\mbf{n})$ refines the\\ $\textsf{CM}_{\mbf{c}}(G(m,d,n))$ partition;
\item the partition of $\textsf{Irr} \, G(m,d,n)$ into Rouquier families associated to $\mc{H}_{G(m,d,n)}(\mbf{n})$ equals the\\ $\textsf{CM}_{\mbf{c}}(G(m,d,n))$ partition for generic values of the parameter $\mbf{c}$. 
\end{enumerate}
\end{thm}

\begin{proof}
It is shown in \cite[Theorem 3.10]{Chlou2} that if $\underline{\lambda}$ is a $d$-stuttering $m$-multi-partition of $n$ such that $\{ \underline{\lambda} \}$ is a Rouquier family for $\mc{R}(y) \mc{H}_{G(m,1,n)}(\mbf{n})$ then the sets $\{ (\underline{\lambda}, \epsilon) \}$, $\epsilon \in C^\vee_{d}$, are Rouquier families for $\mc{R}(y) \mc{H}_{G(m,1,n)}(\mbf{n})$. This agrees with Theorem \ref{thm:mainresult} (1). The second part of \cite[Theorem 3.10]{Chlou2} shows that if $\mc{P}$ is a Rouquier family for $\mc{R}(y) \mc{H}_{G(m,1,n)}(\mbf{n})$ not of the type just described then, in the notation of Theorem \ref{thm:compare1}, $\Gamma( \mc{P})$ is a Rouquier family for $\mc{R}(y) \mc{H}_{G(m,d,n)}(\mbf{n})$. The result \cite[Corollary 3.13]{Mo} shows that the partition of $\textsf{Irr} \, G(m,1,n)$ into Rouquier families associated to $\mc{H}_{G(m,1,n)}(\mbf{n})$ refines the $\textsf{CM}_{\mbf{c}}(G(m,1,n))$ partition. Therefore there exists a $\textsf{CM}_{\mbf{c}}(G(m,1,n))$-partition $\mc{Q}$ such that $\mc{P} \subseteq \mc{Q}$. By Theorem \ref{thm:mainresult} (2), $\Gamma(Q)$ is a $\textsf{CM}_{\mbf{c}}(G(m,d,n))$-partition. Thus $\Gamma(\mc{P}) \subseteq \Gamma(\mc{Q})$ implies that the partition of $\textsf{Irr} \, G(m,d,n)$ into Rouquier families refines the $\textsf{CM}_{\mbf{c}}(G(m,d,n))$ partition.\\
Now let $\mbf{c}$ be a generic parameter for the rational Cherednik algebra associated to $G(m,d,n)$. We think of $\mbf{c}$ as a parameter for the rational Cherednik algebra associated to $G(m,1,n)$. Thus it is a generic point of the subspace defined by $c_j = 0$ for all $j \not\equiv 0 \, \mod \, d$. Correspondingly, $\mbf{n}$ is a generic point in the sublattice of $\Z^{m+1}$ defined by the equations $n_{\mc{S},i + kp} - n_{\mc{S},i + lp} = 0$ for $0 \le i \le p-1$ and $0 \le k < l \le d-1$. We wish to show that the Calogero-Moser partition of $\textsf{Irr} \, G(m,d,n)$ equals the partition of $\textsf{Irr} \, G(m,d,n)$ into Rouquier families. As explained in the previous paragraph, \cite[Theorem 3.10]{Chlou2} and Theorem \ref{thm:mainresult} imply that it suffices to show that the Calogero-Moser partition of $\textsf{Irr} \, G(m,1,n)$ for $\mbf{c}$ equals the partition of $\textsf{Irr} \, G(m,1,n)$ into Rouquier families for $\mbf{n}$. The proof of Lemma \ref{lem:generic} shows that $\underline{\lambda}, \underline{\mu} \in \mc{P}(m,n)$ are in the same Calogero-Moser partition of $\textsf{Irr} \, G(m,1,n)$ if and only if 
\bdm
\sum_{j = 0}^{d-1} \textrm{Res}_{\lambda^{i + pj}}(x^e) = \sum_{j = 0}^{d-1} \textrm{Res}_{\mu^{i + pj}}(x^e) \quad \forall \, 0 \le i \le p-1.
\edm
Combining the results Lemma \ref{prop:hyperpartition} and Lemma \ref{sec:residueHecke} shows that $\underline{\lambda}, \underline{\mu} \in \mc{P}(m,n)$ are in the same Rouquier family of $\mc{R}(y)\mc{H}_{G(m,1,n)}(\mbf{n})$ if and only if the same condition holds. 
\end{proof}

\section*{Acknowledgements}

The research described here was done both at the University of Edinburgh with the financial support of the EPSRC and during a visit to the University  of Bonn with the support of a DAAD scholarship. This material will form part of the author's PhD thesis for the University of Edinburgh. The author would like to express his gratitude to his supervisor, Professor Iain Gordon, for his help, encouragement and patience. He also thanks Dr. Maurizio Martino, Dr. Maria Chlouveraki and Professor Ken Brown for many fruitful discussions and Professor C\'edric Bonnaf\'e for pointing out an error in an earlier version of the article.

\bibliographystyle{plain}
\bibliography{biblo}

\end{document}